\documentclass[reqno]{amsart}
\usepackage{amsfonts,amsmath,amssymb}
\usepackage{appendix}

\numberwithin{equation}{section}

\newtheorem{theorem}{Theorem}[section]
\newtheorem{lemma}{Lemma}[section]

\newtheorem{corollary}{Corollary}[section]
\newtheorem{proposition}{Proposition}[section]
\newtheorem{definition}{Definition}[section]
\newtheorem{remark}{Remark}[section]

\newcommand\eps{\varepsilon}
\newcommand\R{{\mathbb{R}}}
\newcommand\C{{\mathbb{C}}}

\begin{document}
\title[Mass critical 4NLS in high dimensions]
{The mass-critical fourth-order Schr\"odinger equation in high dimensions}
\author{Benoit Pausader}
\address{Mathematics Department,
Box 1917,
Brown University,
Providence, RI 02912}
\email{Benoit.Pausader@math.brown.edu}
\author{Shuanglin Shao}
\address{Institute for Advance Study, Princeton, NJ 08540}
\email{slshao@math.ias.edu}
\date{\today}


\begin{abstract}
We prove global wellposedness and scattering for the Mass-critical homogeneous fourth-order Schr\"odinger equation in high dimensions $n\ge 5$, for general $L^2$ initial data in the defocusing case, and for general initial data with Mass less than certain fraction of the Mass of the Ground State in the focusing case.
\end{abstract}

\maketitle

\section{Introduction}
The fourth-order Schr\"odinger equations have been introduced by Karpman \cite{Kar} and Karpman and Shagalov \cite{KarSha} to take into account the role of ``small fourth-order dispersion" in the propagation of intense laser beams in a bulk medium with Kerr nonlinearity. These equations are defined as follows,
\begin{equation}\label{eq:4NLS-general-1}
i\partial_t u+ \Delta u+|u|^{2\sigma}u+\eps\Delta^2 u=0, \quad u: \R\times \R^n\to \C.
\end{equation}
When $\varepsilon=0$, $n=2$ and $\sigma=1$, this corresponds to the canonical model. When $2<n\sigma <4$, Equation \eqref{eq:4NLS-general-1} can be viewed as a combination of $L^2$ super-critical second-order Schr\"odinger and sub-critical fourth-order Schr\"odinger equation. In this case, Karpman and Shagalov \cite{KarSha} showed, among other things, that the waveguides induced by the nonlinear Schr\"odinger equation become stable when $\vert\varepsilon\vert$ is taken sufficiently large. Then Equation \eqref{eq:4NLS-general-1} is predominantly governed by the corresponding fourth-order equation,
\begin{equation}\label{eq:4NLS-general-2}
i\partial_tu+|u|^{2\sigma}u+\eps\Delta^2u=0.
\end{equation}
Under a suitable change of variable in time, the $L^2$-critical, or Mass-critical, homogeneous case of these equations is
\begin{equation}\label{4NLS}
i\partial_tu+\Delta^2u+\lambda\vert u\vert^\frac{8}{n}u=0,
\end{equation}
where $\lambda=\pm 1$. For $\lambda=1$, it is called ``defocusing" while focusing for $\lambda=-1$. The terminology ``Mass-critical" is due to the fact that both the Mass $M(u)$ defined by
\begin{equation}\label{DefOfMass}
M(u):=\int_{\mathbb{R}^n}\vert u(t,x)\vert^2dx
\end{equation}
and the equation itself are invariant under the rescaling symmetry $$u(t,x)\mapsto \lambda^{n/2}u(\lambda^4 t,\lambda x)$$ for $\lambda>0$. Note that the Mass is conserved by the flow, hence we do not specify time in the notation.

Equation \eqref{eq:4NLS-general-2} has been recently investigated in Fibich, Ilan, and Papanicolaou \cite{FibIlaPap}. They showed that, when $0<\sigma n<4$, any initial data in $L^2$ gives rise to a global solution. In the $L^2$-critical case we discuss here, much less is known. Numerics suggest that in the focusing case $\lambda<0$, there exist solutions that blow up in finite time, while it is conjectured that in the defocusing case any initial data with finite mass leads to a global solution. In this paper, we give a partial positive answer to this question.

Semilinear fourth order Schr\"odinger equations similar to \eqref{4NLS} have been widely investigated. Fibich, Ilan and Papanicolaou \cite{FibIlaPap} give general results of wellposedness in $H^2$. Pausader \cite{Pau2}  and Segata \cite{Seg3} study the cubic case. For the Energy-critical case with nonlinearity given by $F(u)=\vert u\vert^{8/(n-4)}u$, we refer to Miao, Xu and Zhao \cite{MiaXuZha2} and Pausader \cite{Pau1,Pau2}  for the defocusing case and Miao, Xu and Zhao \cite{MiaXuZha} and Pausader \cite{PauFoc} for the focusing case with radially symmetrical initial data. For the Mass-critical case we discuss here, we refer to Chae, Hong and Lee \cite{ChaHonLee} for a result about the concentration of blow-up solutions. In \cite{JiangPauSha}, Jiang-Pausader-Shao were able to establish a precise linear profile decomposition analogous to that in \cite{Shao2} which takes into account the frequency parameter.

A related equation also appears in the study of the motion of a filament of vortex in an inviscid fluid as in Fukumoto and Mofatt \cite{FukMof}, Huo and Jia \cite{HuoJia,HuoJia2} and Segata \cite{Seg1,Seg2}.

The question of global wellposedness and scattering for Mass-critical or Energy critical Schr\"odinger (NLS) equations (Equation \eqref{eq:4NLS-general-1} with $\eps =0$ for suitable $\sigma$) has been the subject to many works recently, most notably by Bourgain \cite{Bourgain:1999:radial-NLS}, Colliander, Keel, Staffilanni, Takaoka and Tao \cite{ColKeeStaTakTao}, Kenig and Merle \cite{KenMer} and Killip and Visan \cite{KilVis} for the Energy-critical case and Tao, Visan and Zhang \cite{TaoVisZha2}, Killip, Tao and Visan \cite{KilTaoVis} for the Mass-critical case. We refer readers to the survey by Killip and Visan \cite[p. 6-8]{KilVis2} for a detailed account. In the Mass-critical context, for radial initial data in dimensions $n\ge 2$, the authors in \cite{TaoVisZha2, KilTaoVis, KilVisZha} were able to establish global well-posedness and scattering for second-order Schr\"odinger equation.  In this paper, we investigate the analogous question for \eqref{4NLS}. Our first result in this paper asserts that in the defocusing case in high dimensions, global wellposedness and scattering hold for equation \eqref{4NLS} even for nonradial initial data.
\begin{theorem}\label{MainThm}
Let $n\ge 5$ and $\lambda=1$. Then for any initial data $u_0\in L^2$, there exists a unique solution of \eqref{4NLS} $u\in C(\mathbb{R},L^2)\cap L^{\frac{2(n+4)}{n}}(\mathbb{R}\times\mathbb{R}^n)$ such that $u(0)=u_0$. Besides this solution scatters in the sense that there exist two elements $\omega^\pm\in L^2$ such that
\begin{equation}\label{ScatStatement}
\Vert u(t)-e^{it\Delta^2}\omega^\pm\Vert_{L^2}\to 0
\end{equation}
as $t\to \pm\infty$.
\end{theorem}
As a consequence of our analysis, we prove that sequences of nonlinear solutions with bounded Mass have a well understood loss of compactness which is only due to the symmetries \eqref{DefTau} of the equation. We refer to Theorem \ref{LossOfCompThm} in Section \ref{Sec-3Sce} for a precise statement.

\begin{remark}
Thanks to the stronger dispersion for the fourth order Schr\"odinger equation, we can take off the radial assumption of the initial data.  This is in contrast with the results on the Mass-critical Schr\"odinger equation developed in Killip, Tao and Visan \cite{KilTaoVis}, Killip, Visan and Zhang \cite{KilVisZha} and Tao, Visan and Zhang \cite{TaoVisZha2} where any global wellposedness result for $L^2$-data holds only in the case of radially symmetrical data in dimensions at least two.
\end{remark}
\begin{remark}\label{re:reduction-to-an-inequ}
By a well-known reduction (see, e.g., Tao \cite{TaoBook}), global well-posedness and scattering \eqref{ScatStatement} reduce to the following \textit{a priori} bound
\begin{equation}\label{eq:reduction-to-an-inequ}
\|u\|_{ L^{\frac {2(n+4)}{n}}(\R\times \R^n)}\le C\bigl(\|u\|_{L^2}\bigr),
\end{equation}
where the constant only depends on the mass of the initial data, and in the focusing case, we also require that $M(u)<M(Q)$.
\end{remark}

Equation \eqref{4NLS} differs from its Mass-critical second-order NLS in several ways. First the dispersion relation reads $\omega(k)=\vert k\vert^4,$
which implies that high frequency waves move much faster than low frequency ones.\footnote{Indeed frequency $k$ has speed $v=4\vert k\vert^2k$ instead of $2k$.} So that, as is manifest from \eqref{StricEst}, we are able to gain some regularity. Secondly, Equation \eqref{4NLS} lacks Galilean invariance. Roughly speaking, this lack means that, unlike \eqref{DefTau}, the frequency modulation $u_0(x)\mapsto e^{ix\xi_0}u_0(x)$, for any  $\xi_0\neq 0\in \R^n$, is not a symmetry of the equation.
This affects us in two ways. On the one hand, the linear profile decomposition as in Lemma \ref{le:linear-profile} does not take this parameter into account, which makes the process of renormalization and search for minimal-mass counterexamples easier. On the other hand, even though the Momentum of solutions is conserved, we cannot set it to be zero as is the case for the second NLS, see e.g. Duyckaerts, Holmer and Roudenko \cite{DuyHolRou} or Killip and Visan \cite{KilVis2}. This requires that we deal with standing waves {\it and} travelling waves in Section \ref{SecFoc}.

In the focusing case $\lambda<0$, one cannot hope for such a global result as in Theorem \ref{MainThm}. Indeed, the existence of nontrivial solutions of the elliptic equation
\begin{equation}\label{EllipticEquation}
\Delta^2Q+ Q=\vert Q\vert^\frac{8}{n}Q
\end{equation}
provides solutions $u(t,x)=e^{-it}Q(x)$ which clearly do not scatter.\footnote{A solution of \eqref{EllipticEquation} with minimal $L^2$ norm, $\sqrt{M(Q)}$ is called a Ground State for this equation.} In fact in this case it is suspected (see e.g., the numerical work in Fibich, Ilan and Papanicolaou \cite{FibIlaPap})  that there exist smooth initial data which do not lead to a global solution. However, as proved in Fibich, Ilan and Papanicolaou \cite{FibIlaPap} and as is also similar to the case of the second order NLS (see Weinstein \cite{Wei,Wei2}), $H^2$-solutions with Mass $M(u)$ strictly smaller than the Mass of the Ground State $M(Q)$ are global. Hence in this case it is natural to expect that they also scatter.
Our other main results are concerned with this question. First, we give a positive answer in the case of high dimensions, when the solution is radially symmetrical. More precisely, we prove the following

\begin{theorem}\label{FocusingRadialThm}
Let $n\ge 5$ and $\lambda=-1$. Then for any radially symmetrical initial data $u_0\in L^2$ such that $M(u)<M(Q)$, there exists a unique solution of \eqref{4NLS} $u\in C(\mathbb{R},L^2)\cap L^{\frac{2(n+4)}{n}}(\mathbb{R}\times\mathbb{R}^n)$ such that $u(0)=u_0$. This solution scatters in the sense that there exist two elements $\omega^\pm\in L^2$ such that \eqref{ScatStatement} holds true as $t\to\pm\infty$.
\end{theorem}
Finally in the case of general $L^2$-data, if we do not assume radial symmetry anymore, we are also able to prove that there is still a positive threshold under which long time existence,
uniqueness and scattering hold true.
\begin{theorem}\label{NonradialFocusingTheorem}
Let $n\ge 5$ and $\lambda=-1$. For any initial data $u_0\in L^2$ of Mass smaller than
\begin{equation}\label{MAST}
M_\ast=\left(\frac{1}{4}\right)^\frac{n}{8}M(Q)
\end{equation}
there exists a unique global solution of \eqref{4NLS}, $u\in C(\mathbb{R},L^2)\cap L^{\frac{2(n+4)}{n}}(\mathbb{R}\times\mathbb{R}^n)$ such that $u(0)=u_0$. Besides this solution scatters.
\end{theorem}
Again, we are able to describe the loss of compactness for solutions of Mass below $M_\ast$ and $M(Q)$ in the radially symmetrical setting.

Although Theorem \ref{NonradialFocusingTheorem} does not prove global wellposedness and scattering all the way up to the Mass of the Ground State, we want to emphasize that it is not a result about small data.
The (probable) nonoptimality of the bound on the Mass comes from the fact that in our estimate the norm of the gradient appears, which we are not able to connect to some quantity related to the equation.

We now briefly sketch our arguments. By the reduction in Remark \ref{re:reduction-to-an-inequ}, we prove the {\it a priori} inequality \eqref{eq:reduction-to-an-inequ} by contradiction. To this end, we first reset the problem as a variational problem and wish to extract an extremal. This is achieved by using a concentration-compactness approach as in \cite{KerNew, KenMer}. The important ingredients in this step are the linear profile decomposition, which aim to compensate for the defect of compactness of the solution operator from $L^2$ to the Strichartz space,  and the stability lemma, which is concerned with constructing true solutions from approximate ones under suitable smallness conditions. After being renormalized by the natural symmetries associated to Equation \eqref{4NLS}, a minimal-mass blow-up solution is exhibited which can be further classified into one of three possible scenarios: self-similar (finite-time blow-up solutions), double high-to-low cascade, and soliton solutions. See Theorem \ref{3ScenariosThm} for a precise statement. Essentially we are able to prove that the minimal element enjoys more regularity, which enables us to conclude that it has a finite scattering norm, which leads to a contradiction to our assumption that \eqref{eq:reduction-to-an-inequ} fails.

Within this scheme of proof by contradiction, we will expand a little more on gaining regularity and disproving each scenario. In the self-similar case, this additional control comes from the stronger dispersion, which thus gives rise to a vanishing effect of the linear part of solutions. We remark that this exactly allows us to get rid of the assumption of radial symmetry. For the remaining two scenarios,  we use a ``Double-Duhamel'' formula introduced by Tao \cite{Tao} to gain regularity. It is the ``Double-Duhamel'' argument that imposes the restriction $n\ge 5$ for two reasons. First, we need the linear propagator to be integrable in time. Secondly, this argument allows us to gain $n/2$-derivatives which we want to be bigger than $2$.

Having enough regularity on the minimal-Mass blow-up solutions, the combination of conservation of Mass and Energy and the sharp Gagliardo-Nirenberg inequality allows us to exclude them when they change scale, i.e. in the self-similar and cascade cases. Thus we are left with dealing with a Soliton which is either a standing wave or a traveling wave. In the defocusing case, since we are in dimensions $n\ge 2$, we can exclude them simultaneously by a Virial-type identity in a direction orthogonal to that of the velocity, and this concludes the proof. In the focusing case, the Virial-type identity is weaker and it only excludes standing waves (i.e. waves with $0$ Momentum). This, however, is sufficient to treat the case of radially symmetrical data and gives Theorem \ref{FocusingRadialThm}. In the case of traveling waves, the Virial-type identity merely gives us a control on the velocity of the solution. Then we use an (interaction-Virial-type) estimate to control the dispersion of the Mass, which gives us a new relation between the Mass, the Momentum, the Energy and the current of Mass. At this point, we use an inequality inspired by Banica \cite{Ban} to understand the loss of optimality of the Gagliardo-Nirenberg inequality for solutions with nonzero Momentum. This gives us a relation between the Mass, the Momentum, the Energy, the current of Mass and the norm of the gradient. Finally we control the norm of the gradient by the Energy and the Massand get Theorem \ref{NonradialFocusingTheorem} with the bound \eqref{MAST}.

This paper is organized as follows: we fix our notations and review some preliminary results in Section \ref{Sec-Not}. In Section \ref{Sec-3Sce}, following an approach introduced by Kenig and Merle \cite{KenMer} and developed by Killip, Tao and Visan \cite{KilTaoVis}, we study the loss of compactness for the nonlinear solutions of \eqref{4NLS}, which subsequently reduces the proof of Theorems \ref{MainThm}, \ref{FocusingRadialThm}  and \ref{NonradialFocusingTheorem} to disproving the existence of some solutions whith special properties. We start the proof of Theorem \ref{MainThm} in Section \ref{Sec-Abs} where, using ideas from Killip, Tao and Visan \cite{KilTaoVis}, we derive abstract results on gain of regularity. Then in Section \ref{Sec-Proof} we exclude the three scenarios in the defocusing case. This finishes the proof of Theorem \ref{MainThm}. In Section \ref{SecFoc}, we apply our analysis to the focusing equation to prove Theorems \ref{FocusingRadialThm} and \ref{NonradialFocusingTheorem}.

\section{Notations and some preliminary results}\label{Sec-Not}

In this section, we introduce some notations. We write $X\lesssim Y$ whenever there exists some constant $C$, possibly depending on the dimension $n$ or on $\lambda$ so that $X\le CY$. Similarly we write $X\simeq Y$ when $X\lesssim Y\lesssim X$. A notation like $\lesssim_a,\simeq_a$ means that the constants in the inequalities may depend on $a$.

We define the Lebesgues spaces on space-time, $L^p(\mathbb{R},L^q)$ or $L^p(L^q)$, as the completion of the space of step functions (functions whose image takes a finite number of values) with respect to the norm
$$\Vert u\Vert_{L^p(\mathbb{R},L^q)}=\left(\int_\mathbb{R}\left(\int_{\mathbb{R}^n}\vert u(t,x)\vert^qdx\right)^\frac{p}{q}dt\right)^\frac{1}{p},$$
with the usual modification when $p$ or $q$ is infinite, and for $I\subset\mathbb{R}$ an interval, we let $L^p(I,L^q)$ be the set of restriction to $I$ of functions in $L^p(L^q)$. When $p=q$, we sometime write it $L^p(I\times\mathbb{R}^n)$ or, if $I=\mathbb{R}$, $L^p_{t,x}$. We are specially interested in the following space-time norms,
\begin{equation}\label{DefinitionOfNorms}
\begin{split}
\Vert u\Vert_{S^0(I)}&=\sup_{(p,q)}\Vert \vert\nabla\vert^\frac{2}{p}u\Vert_{L^p(I,L^q)},\\
\Vert u\Vert_{S^k(I)}&=\Vert \vert\nabla\vert^ku\Vert_{S^0(I)},\\
\Vert u\Vert_{Z(I)}&=\Vert u\Vert_{L^\frac{2(n+4)}{n}(I,L^\frac{2(n+4)}{n})},\\
\Vert u\Vert_{N(I)}&=\Vert \vert\nabla\vert^{-\frac{n}{n+4}}u\Vert_{L^\frac{2(n+4)}{n+8}(I,L^\frac{2(n+4)}{n+6})},
\end{split}
\end{equation}
where the supremum in the first norm is taken over all $S$-admissible values, $(p,q)$, that is all $2\le p,q\le\infty$ such that $(p,q)\ne (2,\infty)$ and $$\frac{2}{p}+\frac{n}{q}=\frac{n}{2}.$$
When $I=\mathbb{R}$, we may omit it in the notation of the norms. We also let $S^0(\mathbb{R})$ the completion of Schwartz functions \footnote{all the space-time derivatives are Schwartz in spatial space, and locally uniformly in time.} under the $S^0$-norm. For $I\subset \mathbb{R}$ an interval, we let $S^0(I)$ be the set of restrictions to $I$ of elements in $S^0(\mathbb{R})$, and $S^0_{loc}(I)$ the set of functions $f$ such that $f\in S^0(J)$ for all compact intervals $J\subset I$. We adopt similar conventions for $Z$ and $N$.

\begin{remark}
The $S^0$ and $Z$-norms are left invariant by the rescaling transformation $\tau$ defined for any $\lambda>0,\,x_0\in \R^d, \, t_0\in \R$ by \begin{equation}\label{DefTau}
\bigl[\tau_{(\lambda,t_0,x_0)}u\bigr](t,x):=\lambda^\frac{n}{2} u(\lambda^4 (t-t_0),\lambda (x-x_0)).
\end{equation}
This transforms a solution $u$ of \eqref{4NLS} with initial data $u(0)=u_0$ to another solution with data at time $t_0$ given by
\begin{equation}\label{DefG}
\tau_{(h,t_0,x_0)}u(t_0)=g_{(h,x_0)}u_0=h^\frac{n}{2}u_0(h(x-x_0))
\end{equation}
with same Mass.
\end{remark}

Before introducing nonlinear solutions, we recall some facts about the linear propagator
\begin{equation*}
e^{it\Delta^2}=\mathcal{F}^{-1}e^{it\vert\xi\vert^4}\mathcal{F},
\end{equation*}
where $\mathcal{F}$ stands for the Fourier transform given by
$$\mathcal{F}u(\xi)=\hat{u}(\xi)=\frac{1}{(2\pi)^\frac{n}{2}}\int_{\mathbb{R}^n} e^{-i\langle x,\xi\rangle}u(x)dx.$$
The linear propagator satisfies the following decay estimate which follows from application of the stationary phase method, cf Stein \cite{Stein:large},
\begin{equation}\label{DecayEstimate}
\Vert P_{N}e^{it\Delta^2}\delta\Vert_{L^\infty}\lesssim N^{-n}t^{-\frac{n}{2}},
\end{equation} where $P_N$ denotes the littlewood-Paley operator defined in \eqref{DefLitPalOp} below.
Using also the trivial estimate $\Vert e^{it\Delta^2}P_N\delta\Vert_{L^\infty}\lesssim N^n$, and summing over all frequencies, we deduce the coarser decay estimate
\begin{equation}\label{DecayEstimate2}
\Vert e^{it\Delta^2}\delta\Vert_{L^\infty}\lesssim t^{-\frac{n}{4}}.
\end{equation}
A deeper consequence is the following Strichartz estimates from Pausader \cite{Pau1} using an abstract result of Keel and Tao \cite{KeeTao} (see also Kenig, Ponce and Vega \cite{KenPonVeg} for previous results).
\begin{equation}\label{StricEst}
\Vert u\Vert_{S^0(I)}\lesssim \Vert u_0\Vert_{L^2}+\Vert h\Vert_{N(I)}
\end{equation}
whenever $u\in S^0(I)$ is a solution of the linear equation
$$i\partial_tu+\Delta^2u=h$$
such that $u(t_0)=u_0\in L^2$ for some $t_0\in I$ and $h\in N(I)$.

\begin{definition}[Strong solutions]
Let $I\subset\mathbb{R}$ be an interval. A strong solution of \eqref{4NLS} on $I$ is a function $u\in S_{loc}^0(I)$ satisfying the Duhamel formula: for all $t, t_0\in I$,
\begin{equation}\label{DuhamelFormula}
u(t)=e^{it\Delta^2}u(t_0)+i\lambda\int_{t_0}^te^{i(t-s)\Delta^2}\left(\vert u\vert^\frac{8}{n}u(s)\right)ds.
\end{equation}
\end{definition}
Note that, by Strichartz estimates \eqref{StricEst}, each term make sense as continuous function in $L^2$. Strong solutions have a conserved Mass $M(u)$ as defined in \eqref{DefOfMass}, and when they are smoother they enjoy other conserved quantities. Here we use the conservation of Momentum for $H^1$ solutions $u\in C(I,H^1)$, where the Momentum vector is defined as
\begin{equation}\label{DefOfMomentum}
\hbox{Mom}(u)=\hbox{Im}\int_{\mathbb{R}^n}u(t,x)\nabla \bar{u}(t,x)dx
\end{equation}
and conservation of Energy for $H^2$ functions $u\in C(I,H^2)$, where the Energy is given by
\begin{equation}\label{DefOfEnergy}
E(u)=\frac{1}{2}\int_{\mathbb{R}^n}\vert\Delta u(t,x)\vert^2dx+\frac{n\lambda}{2(n+4)}\int_{\mathbb{R}^n}\vert u(t,x)\vert^\frac{2(n+4)}{n}dx.
\end{equation}
Note that in these notations, we omit $t$ since these are conserved quantities.

We also need the sharp Gagliardo-Nirenberg inequality from Fibich, Ilan and Papanicolaou \cite{FibIlaPap}
\begin{equation}\label{Gag}
\Vert f\Vert_{L^\frac{2(n+4)}{n}}^\frac{2(n+4)}{n}\le \frac{n+4}{n}\left(\frac{M(f)}{M(Q)}\right)^\frac{8}{n}\Vert \Delta f\Vert_{L^2}^2
\end{equation}
for all functions $f\in H^2$, where $Q$ is a Ground State.

Besides a consequence of Strichartz estimates \eqref{StricEst} is the following local well-posedness statement.
\begin{proposition}\label{LocExProp}
Let $n\ge 1$, $\lambda=\pm 1$. Then, for any initial data $u_0\in L^2$, there exists an interval $I\subset \mathbb{R}$ containing a neighborhood of $0$ and a unique function $u\in S^0(I)$ solution of \eqref{4NLS} such that $u(0)=u_0$. This solution has conserved Mass.
Besides if $u_0\in H^2$, $u\in S^2(I)$, $u$ has conserved Energy and Momentum and $\vert I\vert\gtrsim_{M(u)} \Vert\Delta u_0\Vert_{L^2}^{-2}$. In particular, in the defocusing case $u$ can be extended to a solution on $\mathbb{R}$ in the sense that $u\in S^0(J)$ for all compact intervals $J\subset\mathbb{R}$.
Finally, if $u\in S^0(\mathbb{R})$, then $u$ scatters in the sense that \eqref{ScatStatement} holds true as $t\to \pm\infty$.
\end{proposition}
One also easily sees that in case $u\in S^0(\mathbb{R})$ and $u_0\in H^2$, then $u$ scatters to linear solutions $e^{it\Delta^2}\omega^\pm$ such that $M(\omega^\pm)=M(u)$ and $E(u)=\Vert \Delta\omega^\pm\Vert_{L^2}^2$ for all $t\in\mathbb{R}$. In particular, either $u$ has positive Energy or $u=0$.

In our nonlinear analysis, we need some tools from Littlewood-Paley theory. Let $\psi\in C^\infty_c(\mathbb{R}^n)$ be supported in the ball $B(0,2)$, and such that $\psi=1$ in $B(0,1)$. For any dyadic number $N=2^k,k\in\mathbb{Z}$, we define the following
Littlewood-Paley operators:
\begin{equation}\label{DefLitPalOp}
\begin{split}
&\widehat{P_{\leq N}f}(\xi)=\psi(\xi/N)\hat{f}(\xi),\\
&\widehat{P_{>N}f}(\xi)=(1-\psi(\xi/N))\hat{f}(\xi),\\
&\widehat{P_Nf}(\xi)=\left(\psi(\xi/N)-\psi(2\xi/N)\right)\hat{f}(\xi).
\end{split}
\end{equation}
Similarly we define $P_{<N}$ and $P_{\ge N}$ by the equations
$$P_{<N} = P_{\leq N}-P_N\hskip.2cm\hbox{and}\hskip.2cm P_{\ge N} = P_{> N} + P_N.$$
These operators commute one with another. They also commute with derivative operators and with the semigroup
$e^{it\Delta^2}$. In addition they are self-adjoint and bounded on $L^p$ for all $1\le p\le\infty$.
Moreover, they enjoy the following Bernstein inequalities:
\begin{equation}\label{BernSobProp}
\begin{split}
&\hskip.8cm \Vert P_{\ge N}f\Vert_{L^p}\lesssim_s
N^{-s}\Vert\vert\nabla\vert^sP_{\geq N}f\Vert_{L^p}\lesssim_s N^{-s}\Vert\vert\nabla\vert^sf\Vert_{L^p}\\
&\hskip.8cm\Vert\vert\nabla\vert^sP_{\le N}f\Vert_{L^p}\lesssim_s N^s\Vert P_{\le
N} f\Vert_{L^p}
\lesssim_s N^s\Vert f\Vert_{L^p}\\
&\hskip.8cm\Vert \vert\nabla\vert^{\pm s}P_Nf\Vert_{L^p}\lesssim_s
N^{\pm s}\Vert P_Nf\Vert_{L^p}\lesssim_s N^{\pm s}\Vert f\Vert_{L^p}\\
&\hskip.8cm \Vert P_Nf\Vert_{L^q}\lesssim N^{\frac{n}{p}-\frac{n}{q}}\Vert P_Nf\Vert_{L^p}
\end{split}
\end{equation}
for all $s\ge 0$, and all $1\le p\le q\le\infty$, independently of $f$, $N$, and $p$, where
$\vert\nabla\vert^s$ is the classical fractional differentiation operator. We refer to Tao \cite{TaoBook} for more details. Finally, we let $\alpha=2(n+5)/(n+4)$ and $F(u)=\lambda\vert u\vert^\frac{8}{n}u$.


\section{Existence of minimal-mass blow-up solutions, and three scenarios}\label{Sec-3Sce}

In this section, we establish a linear profile decomposition for solutions to the linear equation of \eqref{4NLS}, namely
$$iu_t+\Delta^2 u=0.$$
It roughly asserts that, a sequence of linear solutions with bounded initial data in $L^2$, after passing to a subsequence if necessary, can be rewritten as a sum of a superposition of profiles and an error term. The profiles are ``orthogonal"  and the error term is small in the Strichartz norm, see Lemma \ref{le:linear-profile} and the following remark. The purpose of linear profile is to compensate for the defect of compactness of the solution operator $e^{it\Delta^2}, L^2\to Z(\mathbb{R})$. With it, we are able to extract a minimal-Mass blow-up solution to \eqref{4NLS} if it blows-up in the $Z$-norm sense. Furthermore, as an extremal case, this minimal element enjoys good ``compactness" properties. More precisely, our main result in this section is the following.

Let $M_{max}$ (resp. $M_{max}^{rad}$) be the first Mass-level for which there exists solutions (resp. radially symmetrical solutions) of arbitrarily large $Z$-norm. See the end of Subsection \ref{SubSec-StabLemma} for a more precise definition. Then we have the following.
\begin{theorem}\label{3ScenariosThm}
Suppose that $M_{max}<+\infty$. Then there exists $u\in S^0_{loc}(I)$ a maximal-lifespan strong solution of Mass equal to $M_{max}$, such that $$\Vert u\Vert_{Z(I)}= +\infty.$$
Furthermore, we have the following compactness property: there exist two smooth functions $N: I \to \mathbb{R}_+^\ast$ and $y: I \to \mathbb{R}^n$ such that
\begin{equation}\label{CompactnessImage}
K=\{v(t)=g(t)^{-1}u(t)=g_{(N(t),y(t))}^{-1}u(t):t\in I\}\hskip.5cm\hbox{is precompact in }L^2
\end{equation}
and one of the following three scenarios holds true:
\begin{itemize}
\item[I.] (Self-similar solution) There holds $I=(0,+\infty)$ and $N(t)=t^{-\frac{1}{4}}$ for all $t$.
\item[II.] (Double high-to-low cascade) There holds $I=\mathbb{R}$, $\liminf_{t\to\pm \infty}N(t)=0$,
and $N(t)\le 1$ for all $t$.
\item[III.] (Soliton-like solution) There holds $I=\mathbb{R}$ and $N(t)=1$ for all $t$.
\end{itemize}
Furthermore if $n\ge 2$ and $M^{rad}_{max}<+\infty$, then the same conclusion holds true with the additional information that $u$ is radially symmetrical.\footnote{In that case, one could also assume that $y(t)=0$, although we do not use it here.}
\end{theorem}

We call the function $g(t)=g_{(N(t),y(t))}$ appearing in \eqref{CompactnessImage} and defined in \eqref{DefG}, the rescaling function of $u$.

\begin{remark}
By precompactness of $K$ (see also Corollary \ref{SmoothNy} below), we may also assume that $N(t)$ and $y(t)$ satisfy the following relations
\begin{equation}\label{VariationOfH}
\begin{split}
&\vert N^{-5}\dot{N}\vert\lesssim_u 1,\hskip.1cm\hbox{ and }\hskip.1cm\vert N^{-3}\dot{y}\vert\lesssim_u1.
\end{split}
\end{equation}
Note that the first and last scenarios correspond to saturating the first inequality.
\end{remark}

The remaining part of this section is devoted to proving Theorem \ref{3ScenariosThm} which is the analog for \eqref{4NLS} of Theorem 1.16 in Killip, Tao and Visan \cite{KilTaoVis}.


\subsection{A linear profile decomposition.}
As emphasized in the introduction, the profile decomposition is an important ingredient to extract the blow-up bubbles for certain critical equations such as the Schr\"odinger, wave and (generalised) Korteweg-de Vries equations. Bahouri and G\'erard \cite{BahGer} established a decomposition for the  energy-critical wave equation in $\R^3$. Keraani \cite{Ker} treated the case for the Energy-critical Schr\"odinger equation. For the Mass-critical Schr\"odinger equation, Merle and Vega \cite{MerVeg} first established the linear profile decomposition in spirit similar to Bourgain \cite{Bou}. We also refer to B\'egout and Vargas \cite{BegVar}, Carles and Keraani \cite{CarKer}, G\'erard \cite{Ger}, Keraani \cite{KerNew}, and Shao \cite{Shao1,Shao2} for related works based on profile decomposition.

In the following, we define a scale-core to be a sequence $(h_k,t_k,x_k)\in (0,\infty)\times \mathbb{R}\times\R^n$. We call two scale-cores $(h_k^j,t_k^j,x_k^j)_k$ and $(h_k^p,t_k^p,x_k^p)_k$ orthogonal if
\begin{equation}\label{eq:ortho-parameter}
\lim_{k\to +\infty} \left(\frac {h_k^j}{h_k^p}+\frac{h_k^p}{h_k^j} +(h_k^j)^4|t_k^j-t_k^p|+h_k^j|x_k^j-x_k^p|\right)=+\infty.
\end{equation}

The main result of this subsection is the following linear profile decomposition.
\begin{lemma}[Linear profile decomposition]\label{le:linear-profile}
Let $(u_k)_{k\ge 1}$ be a sequence of functions satisfying $\|u_k\|_{L^2}\le 1$.  Then up to a subsequence, there exists a sequence of $L^2$ functions $(\phi^j)_{j\ge 1}$ and a family of pairwise orthogonal scale-cores $(h_k^j,t_k^j,x_k^j)\in (0,\infty)\times \mathbb{R}\times\R^n$ such that, for any $l\ge 1$, there exists a function $w_k^l\in L^2$ satisfying
\begin{equation}\label{eq:prof}
e^{it\Delta^2}u_k=\sum_{1\le j\le l} \tau_{(h_k^j,t_k^j,x_k^j)}e^{i(\cdot)\Delta^2}\phi^j+e^{it\Delta^2}w_k^l,
\end{equation}
where $\tau$ is defined in \eqref{DefTau} and
\begin{align}
&\label{eq:err} \lim_{l\to +\infty}\limsup_{k\to +\infty} \|e^{it\Delta^2}w_k^l\|_{Z}=0.
\end{align}
\end{lemma}
\begin{remark}
As a consequence of the orthogonality condition \eqref{eq:ortho-parameter}, we have two useful properties for the decomposition, for all $l\ge 1$ and $1\le j\ne p\le l$, there holds that
\begin{align}
&\label{eq:L2-almost-ortho}
\lim_{k\to +\infty} \left(\|u_k\|^2_{L^2}-\bigl(\sum_{j=1}^l\|\phi^j\|^2_{L^2}+\|w_k^l\|^2_{L^2}\bigr)\right)=0\hskip.1cm \hbox{and}\\
 &\label{eq:strichartz-ortho} \lim_{k\to+\infty}\bigl\|\bigl(\tau_{(h_k^j,t_k^j,x_k^j)}e^{it\Delta^2}\phi^j\bigr)\bigl( \tau_{(h_k^p,t_k^p,x_k^p)}e^{it\Delta^2}\phi^p\bigr)\bigr\|_{L^{\frac{n+4}{n}}(\mathbb{R}\times\mathbb{R}^n)}=0.
\end{align}
These are a manifestation that the profiles are ``orthogonal'': they are either separated in space or in time, or they have very different scales. For a similar proof, see Merle and Vega \cite{MerVeg} and Shao \cite{Shao2}.
\end{remark}

We present a short proof of this lemma based on the approach in Killip and Visan \cite{KilVis2}, see also Shao \cite{Shao1}; we also refer readers to Pausader \cite{PauThes} for a slightly different approach based on Bahouri and G\'erard \cite{BahGer} and G\'erard, Meyer and Oru \cite{GerMeyOru}. We start with a refinement of the usual Strichartz inequality. It is well-known that the Strichartz inequality
\begin{equation}\label{eq:strichartz-inequa}
\|e^{it\Delta^2}u_0\|_{L^{\frac{2(n+4)}{n}}(\mathbb{R}\times\mathbb{R}^n)}\le C_{n} \|u_0\|_{L^2}
\end{equation}
is optimal for Lebesgues spaces; but it is sub-optimal within some scale of Besov-spaces. This improvement opens a door to address the defect of compactness of the solution operator.

\begin{lemma}[Refinement of the Strichartz inequality]\label{le:refinement-strichartz}
\begin{equation}\label{eq:refinement-strichartz-1}
\|e^{it\Delta^2}u_0\|_{Z}\le C_{n} \|u_0\|^{\frac {n}{n+4}}_{L^2}\bigl( \sup_N\|P_Ne^{it\Delta^2}u_0\|_{Z}\bigr)^{\frac 4{n+4}},
\end{equation}
where the supremum is taken over all dyadic integers $N=2^k$. Furthermore,
\begin{equation}\label{eq:refinement-strichartz-2}
\|e^{it\Delta^2}u_0\|_{Z}\le C_{n} \|u_0\|^{\frac {n^2+8n+8}{(n+4)^2}}_{L^2}\bigl( \sup_N N^{-\frac{n}{2}}\|P_N e^{it\Delta^2}u_0\|_{L^{\infty}(\mathbb{R}\times\mathbb{R}^n)}\bigr)^{\frac {8}{(n+4)^2}}.
\end{equation}
\end{lemma}

\begin{proof}
We begin with the proof of \eqref{eq:refinement-strichartz-1} when $n\ge 4$. By the Littlewood-Paley square function inequality in Stein \cite[p. 267]{Stein:large} and the Bernstein property \eqref{BernSobProp}, we have that
\begin{align*}
&\|e^{it\Delta^2} u_0\|^\frac{2(n+4)}{n}_{Z} \simeq \|\bigl(\sum_{N} |e^{it\Delta^2}P_N u_0|^2 \bigr)^\frac{1}{2}\|^\frac{2(n+4)}{n}_{L^\frac{2(n+4)}{n}_{t,x}}\\
&=\iint \bigl(\sum_{M} |e^{it\Delta^2}P_M u_0|^2\bigr)^{\frac{n+4}{2n}}\bigl(\sum_{N} |e^{it\Delta^2}P_N u_0|^2\bigr)^{\frac{n+4}{2n}} dxdt\\
&\lesssim \sum_{M\le N} \iint \bigl|e^{it\Delta^2}P_M u_0\bigr|^{\frac{n+4}{n}}\bigl|e^{it\Delta^2}P_N u_0\bigr|^\frac{n+4}{n}dxdt\\
&\lesssim \sum_{M\le N} \iint |e^{it\Delta^2}P_M u_0||e^{it\Delta^2}P_M u_0|^\frac{4}{n}|e^{it\Delta^2}P_N u_0|^\frac{4}{n}|e^{it\Delta^2}P_N u_0|dxdt\\
&\lesssim \sum_{M\le N} \|e^{it\Delta^2}P_M u_0\|_{L^\frac{2(n+4)}{n}(L^\frac{2(n+4)}{n-2})}\|e^{it\Delta^2}P_N u_0\|_{L^\frac{2(n+4)}{n}(L^\frac{2(n+4)}{n+2})}\\
&\times \|(e^{it\Delta^2}P_M u_0)^\frac{4}{n}\|_{L^{\frac {n+4}{2}}_{t,x}}\|(e^{it\Delta^2}P_N u_0)^\frac{4}{n}\|_{L^{\frac {n+4}{2}}_{t,x}}\\
&\lesssim \left(\sum_{N} N^\frac{2n}{n+4}\|e^{it\Delta^2}P_N u_0\|^2_{L^{\frac {2(n+4)}{n}}(\mathbb{R},L^\frac{2(n+4)}{n+2})}\right)\left(\sup_N \|e^{it\Delta^2}P_N u_0\|_{L_{t,x}^{\frac {2(n+4)}{n}}}\right)^\frac{8}{n}\\
&\lesssim  \left(\sum_{N} \| P_N u_0\|^2_{L^2}\right)\left(\sup_N \|e^{it\Delta^2}P_N u_0\|_{L_{t,x}^{\frac {2(n+4)}{n}}}\right)^\frac{8}{n},
\end{align*}
where we used at the last line the Strichartz inequality \eqref{StricEst} to get
$$\Vert \vert\nabla\vert^\frac{n}{n+4}e^{it\Delta^2}u_0\Vert_{L^\frac{2(n+4)}{n}(\mathbb{R},L^\frac{2(n+4)}{n+2})}\lesssim \Vert u_0\Vert_{L^2}.$$ Then \eqref{eq:refinement-strichartz-1} follows when $n\ge 4$. When $1\le n<4$, we choose to deal with $n=3$ only; other cases are treated similarly.
\begin{align*}
&\|e^{it\Delta^2} u_0\|^{\frac{14}{3}}_{L^{\frac{14}{3}}(\mathbb{R}\times\mathbb{R}^n)} \simeq \|\bigl(\sum_{N} |e^{it\Delta^2}P_N u_0|^2 \bigr)^\frac{1}{2}\|^\frac{14}{3}_{L^\frac{14}{3}_{t,x}}\\
&=\iint \bigl(\sum_{L} |e^{it\Delta^2}P_L u_0|^2\bigr)\bigl(\sum_{M} |e^{it\Delta^2}P_M u_0|^2\bigr)^\frac{2}{3}\bigl(\sum_{N} |e^{it\Delta^2}P_N u_0|^2\bigr)^\frac{2}{3} dxdt\\
&\lesssim \iint \sum_{L,M,N} \bigl(\vert e^{it\Delta^2}P_Lu_0\vert \vert e^{it\Delta^2}P_M u_0\vert \vert e^{it\Delta^2}P_Nu_0\vert\bigr)^\frac{4}{3} \vert e^{it\Delta^2}P_L u_0\vert^\frac{2}{3} dxdt\\
&\lesssim \sum_{L,M,N}\Vert (e^{it\Delta^2}P_Lu_0)(\vert e^{it\Delta^2}P_M u_0\vert \vert e^{it\Delta^2}P_Nu_0\vert)^\frac{1}{3}\Vert_{L^\frac{14}{5}(\mathbb{R}\times\mathbb{R}^n)}\\
&\times \Vert (e^{it\Delta^2}P_Lu_0)(e^{it\Delta^2}P_M u_0) (e^{it\Delta^2}P_Nu_0)\Vert_{L^\frac{14}{9}(\mathbb{R}\times\mathbb{R}^n)}\\
&\lesssim \sup_M\Vert e^{it\Delta^2}P_Mu_0\Vert_{L^\frac{14}{3}_{t,x}}^\frac{5}{3}\sum_{L,M,N}\Vert (e^{it\Delta^2}P_Lu_0)(e^{it\Delta^2}P_M u_0) (e^{it\Delta^2}P_Nu_0)\Vert_{L^\frac{14}{9}(\mathbb{R}\times\mathbb{R}^n)}
\end{align*}
and estimating the term in the sum corresponding to the lowest frequency in $L^\frac{14}{3}(L^\infty)$, the term corresponding to the second lowest frequency in $L^\frac{14}{3}(L^\frac{14}{4})$ and the last term in $L^\frac{14}{3}(L^\frac{14}{5})$, we get, using Bernstein Property \eqref{BernSobProp}
\begin{align*}
&\|e^{it\Delta^2} u_0\|^{\frac{14}{3}}_{L^{\frac{14}{3}}(\mathbb{R}\times\mathbb{R}^n)}\\
&\lesssim \sup_M\|e^{it\Delta^2}P_{M} u_0\|_{L^\frac{14}{3}_{t,x}}^\frac{8}{3}\sum_{L\le M\le N} L^\frac{9}{14}M^\frac{1}{14}\Vert e^{it\Delta^2}P_Mu_0\Vert_{L^\frac{14}{3}(L^\frac{14}{5})}\|e^{it\Delta^2}P_{N} u_0\|_{L^\frac{14}{3}(L^\frac{14}{5})}.
\end{align*} Schur's lemma then gives \eqref{eq:refinement-strichartz-1}. The argument proceeds similarly when $n=1,2$.

To prove \eqref{eq:refinement-strichartz-2}, it suffices to show that
$$\|e^{it\Delta^2} P_Nu_0\|_Z\lesssim \|u_0\|^{\frac {n+2}{n+4}}_{L^2}\bigl(N^{-n/2} \|e^{it\Delta^2} P_Nu_0\|_{L^{\infty}(\mathbb{R}\times\mathbb{R}^n)}\bigr)^{\frac 2{n+4}}.$$
This follows from the following two facts. First, using H\"older's inequality, we get
\begin{equation*}
\|e^{it\Delta^2} P_Nu_0\|_Z\le \|e^{it\Delta^2} P_Nu_0\|^{\frac {n+2}{n+4}}_{L^\frac{2(n+2)}{n}(\mathbb{R}\times\mathbb{R}^n)}\|e^{it\Delta^2} P_Nu_0\|^{\frac 2{n+4}}_{L^{\infty}(\mathbb{R}\times\mathbb{R}^n)},
\end{equation*}
then using Bernstein properties \eqref{BernSobProp},
\begin{equation*}
\begin{split}
\|e^{it\Delta^2} P_N u_0\|_{L^\frac{2(n+2)}{n}(\mathbb{R}\times\mathbb{R}^n)}&\simeq N^{-\frac {n}{n+2}}\||\nabla|^{\frac{n}{n+2}}e^{it\Delta^2} P_Nu_0\|_{L^\frac{2(n+2)}{n}(\mathbb{R}\times\mathbb{R}^n)}\\
&\lesssim N^{-\frac{n}{n+2}} \|P_N u_0\|_{L^2}.
\end{split}
\end{equation*}
Hence the proof of Lemma \ref{le:refinement-strichartz} is complete.
\end{proof}

Then a standard greedy algorithm as used in Lemma 3.2 in Shao \cite{Shao2} gives Lemma \ref{le:linear-profile}, see also Carles and Keraani \cite{CarKer}, hence its proof is omitted. Note that  we can select those three parameters $(h_k^j,x_k^j,t_k^j)$ in Lemma \ref{le:linear-profile} simultaneously.


\subsection{A stability lemma.}\label{SubSec-StabLemma} Roughly speaking, the stability lemma says that, if initial data are close enough, then the solutions will be close.
\begin{lemma}\label{lemStab} For any $B>0$ and $\varepsilon >0$, there exists $\delta>0$ such that if $v$ is an approximate solution of \eqref{4NLS} in the sense that
\begin{equation}\label{ApproxEquation}
i\partial_tv+\Delta^2v+\lambda\vert v\vert^\frac{8}{n}v+e=0
\end{equation}
on some interval $I$ with $0\in I$ satisfying $\Vert v\Vert_{Z(I)}\le B$, and two smallness conditions,
\begin{equation}\label{AddedHypothApproxLemm1}
\Vert e\Vert_{N(I)}\le \delta,
\end{equation}
and
\begin{equation}\label{AddedHypothApproxLemm2}
\Vert e^{it\Delta^2}\left(v(0)-u_0\right)\Vert_{Z(I)}\le\delta
\end{equation} for any $u_0\in L^2$. Then there exists a unique strong solution $u\in S^0(I)$ of \eqref{4NLS} which satisfies
$\Vert u-v\Vert_{Z(I)}\le \varepsilon$ and $\Vert u-v\Vert_{S^0(I)}\lesssim \Vert v(0)-u_0\Vert_{L^2}+\varepsilon$.
\end{lemma}

\begin{proof}[Proof of Lemma \ref{lemStab}]

A standard consequence of Strichartz estimates \eqref{StricEst} is that there exists $\kappa>0$ depending only on $n$ and $\lambda$ such that for any interval $I$ and any $t_\ast$ in the closure of $I$, if $u_0\in L^2$ satisfies
\begin{equation}\label{LocExCond}
\Vert e^{it\Delta^2}u_0\Vert_{Z(I)}\le \kappa
\end{equation}
then there exists a unique strong solution $u\in S(I)$ of \eqref{4NLS} such that
\begin{equation*}
\Vert e^{it_k\Delta^2}u_0-u(t_k)\Vert_{L^2}\to 0\hskip.1cm\hbox{as}\hskip.1cm t_k\to t_\ast.
\end{equation*}
Besides, any strong solution $u$ on a compact time interval $[-T_\ast,T^\ast]$ with finite $Z([-T_\ast,T^\ast])$-norm can be extended to a strong solution on a strictly bigger interval $[-T_\ast-s,T^\ast+s]$ for some $s>0$.

Now, we prove Lemma \ref{lemStab}.
We can assume that $I=[0,\gamma]$ for some $\gamma>0$.
We first claim that there exists $b>0$ and $\delta_0>0$ depending only on $n$ and $\lambda$ such that if $u$ and $v$ are as above and
\eqref{AddedHypothApproxLemm1}, \eqref{AddedHypothApproxLemm2} holds with $B=b$ and $\delta\le \delta_0$, then $u$ exists on $I$,
\begin{equation}\label{Claim1Lem3}
\begin{split}
&\Vert u-v\Vert_{Z(I)}\lesssim \Vert e^{it\Delta^2}(v(0)-u_0)\Vert_{Z(I)}+\Vert e\Vert_{N(I)}\hskip.1cm,\hskip.1cm\hbox{and}\\
&\Vert e+\lambda\left(\vert u\vert^\frac{8}{n}u-\vert v\vert^\frac{8}{n}v\right)\Vert_{N(I)}\lesssim \Vert e^{it\Delta^2}(v(0)-u_0)\Vert_{Z(I)}+\Vert e\Vert_{N(I)}.
\end{split}
\end{equation}
To prove this, let $w=u-v$. Then $w$ satisfies the equation
\begin{equation*}\label{ApproxEquationByW}
i\partial_tw+\Delta^2w+\lambda\left(\vert v+w\vert^\frac{8}{n}(v+w)-\vert v\vert^\frac{8}{n}v\right)-e=0
\end{equation*}
and applying Strichartz estimates, we get, for $t$ such that $[0,t]\subset I$ and $u$ is defined on $[0,t]$,
\begin{equation}\label{StricEstApplyedToW}
\begin{split}
\Vert w\Vert_{Z([0,t])}&\lesssim \Vert e^{it\Delta^2}w(0)\Vert_{Z(I)}+\Vert w\Vert_{Z([0,t])}^\frac{n+8}{n}+\Vert w\Vert_{Z([0,t])}\Vert v\Vert_{Z(I)}^\frac{8}{n}+\Vert e\Vert_{N(I)}\\
&\lesssim \Vert e^{it\Delta^2}w(0)\Vert_{Z(I)}+\Vert e\Vert_{N(I)}+\Vert w\Vert_{Z([0,t])}^\frac{n+8}{n},
\end{split}
\end{equation}
provided that $b$ is sufficiently small.
Besides, $h(t)= \Vert w\Vert_{Z([0,t])}$ is a continuous function of $t$ defined on $I$ satisfying $h(0)=0$. As a consequence of \eqref{StricEstApplyedToW} we then get that if $\delta\le\delta_0$ with $\delta_0$ sufficiently small, then $u$ exists on $I$ and
\begin{equation}\label{ProofOfLem3ControlOfW}
\Vert w\Vert_{Z([0,t])}\lesssim \Vert e^{it\Delta^2}w(0)\Vert_{Z(I)}+\Vert e\Vert_{N(I)}
\end{equation}
for all $t\in I$, and hence using \eqref{ProofOfLem3ControlOfW} and H\"older's inequality, we get \eqref{Claim1Lem3} in this special case. Now, in the general case, let $M=[(B/b)^\frac{n+8}{n}]$, and divide $I=\bigcup_{j=0}^MI_j$ into subintervals $I_j=[\alpha_j,\alpha_{j+1}]$ with
$0\le j\le M$ and $0=\alpha_0\le\alpha_1\le\dots\le\alpha_{M+1}=\gamma$ and such that \eqref{AddedHypothApproxLemm1} holds on $I_j$ with $B=b$.
For $0\le j\le M$, let
\begin{equation*}
\begin{split}
&X_j=\Vert e^{i(t-\alpha_j)}w(\alpha_j)\Vert_{Z(I_j)}\\
&Y_j=\Vert \lambda\left(\vert u\vert^\frac{8}{n}u-\vert v\vert^\frac{8}{n}v\right)-e\Vert_{N(I_j)}\hskip.1cm,\hskip.1cm\hbox{and}\\
&Z_j=\Vert w\Vert_{Z(I_j)}.
\end{split}
\end{equation*}
Then, the special case we have just treated insures that, if $\delta,X_j\le \delta_0$ for some $\delta_0$, then there exists $C_y$ and $C_z$ such that
$Y_j\le C_y\left(X_j+\delta\right)$ and $Z_j\le C_z\left(X_j+\delta\right)$. Besides, using Strichartz estimates, we have that
\begin{equation*}
\begin{split}
&X_{j+1}\\
&=\Vert e^{i(t-\alpha_{j+1})\Delta^2}w(\alpha_{j+1})\Vert_{Z(I_{j+1})}\\
&=\Vert e^{it\Delta^2}w(0)-i\int_0^{\alpha_{j+1}}e^{i(t-s)\Delta^2}\left(\lambda\left(\left(\vert u\vert^\frac{8}{n}u\right)-\left(\vert v\vert^\frac{8}{n}v\right)\right)(s)-e(s)\right)ds\Vert_{Z(I_{j+1})}\\
&\lesssim \Vert e^{it\Delta^2}w(0)\Vert_{Z(I_{j+1})}+\sum_{k=0}^j\Vert \lambda\left(\vert u\vert^\frac{8}{n}u-\vert v\vert^\frac{8}{n}v\right)-e\Vert_{N(I_k)}\\
&\le C_x\left( \delta+\sum_{k\le j}Y_k\right)
\end{split}
\end{equation*}
for some constant $C_x>1$.
Now, let $(x_k)_k$, $(y_k)_k$ and $(z_k)_k$ be defined by $x_0=\delta$ and by
$x_{k+1}=C_x(\delta+\sum_{j\le k}y_j)$, $y_k=C_y\left(x_k+\delta\right)$ and $z_k=C_z\left(x_k+\delta\right)$. Then, we see that for fixed $j$, $x_j,z_j$ are linear functions of $\delta$. Choose $\delta>0$ sufficiently small so that
$$x_M\le \delta_0\hskip.2cm\hbox{and}\hskip.2cm\sum_{j=0}^Mz_j^\frac{2(n+4)}{n}\le \varepsilon^\frac{2(n+4)}{n}.$$ Then, using the discrete Gromwall's lemma, we see that for all $j$, $\Vert u\Vert_{Z(I_j)}$ is a priori bounded, hence $u$ and $w$ exist on all of $I$, and
$$\Vert u-v\Vert_{Z(I)}=\left(\sum_{j=0}^MZ_j^\frac{2(n+4)}{n}\right)^\frac{n}{2(n+4)}\le \varepsilon.$$
Another application of Strichartz estimates finishes the proof.
\end{proof}

A consequence of Lemma \ref{lemStab} is that a nonlinear solution $u\in S_{loc}(I)$ can be extended to a strictly larger interval of existence (as a strong solution) if and only if $u$ has finite $Z$-norm. Thus, it is natural to introduce
$$\Lambda(L)=\sup\{\Vert u\Vert_{Z(I)}^\frac{2(n+4)}{n}:\hbox{for all interval}\hskip.1cm I\hskip.1cm\hbox{and all solutions}\hskip.1cm u\hskip.1cm\hbox{such that}\hskip.1cm M(u)<L\}.$$
A consequence of local wellposedness is that $\Lambda$ is sublinear in a neighborhood of $0$. We also let
$$M_{max}=\sup\{M>0: \Lambda(M)<+\infty\}.$$
In order to treat the radially symmetrical case, we also consider
$$\Lambda^{rad}(L)=\sup\{\Vert u\Vert_{Z(I)}:\hbox{for all}\hskip.1cm I\hskip.1cm\hbox{and all radial solutions}\hskip.1cm u\hskip.1cm\hbox{such that}\hskip.1cm M(u)<L\},$$
and we let
$$M_{max}^{rad}=\sup\{M>0: \Lambda^{rad}(M)<+\infty\}.$$
In view of the discussion above, the goal of this paper is to prove that, when $n\ge 5$, $M_{max}=+\infty$ in the defocusing case and $M_{max}^{rad}=M(Q)$ in the focusing case.


\subsection{Existence of minimal-Mass blow-up solution: Palais-Smale lemma modulo symmetries.}
In this section, we establish the first part of Theorem \ref{3ScenariosThm} by exhibiting an element with minimal Mass and infinite $Z$-norm and proving that \eqref{CompactnessImage} holds.
\begin{lemma}[Palais-Smale lemma modulo symmetries]\label{lem4} For any sequence of nonlinear solutions $(u_k)$ of \eqref{4NLS}
defined on an interval $(-T_k,T^k)$, and any sequence of time $(t_k)$ such that $M(u_k)\to M_{max}$ and
\begin{equation}\label{CondForCompactnessBigSNorm}
\lim_k\Vert u_k\Vert_{Z(-T_k,t_k)}=\lim_k\Vert u_k\Vert_{Z(t_k,T^k)}=+\infty,
\end{equation}
there exists two sequences $(N_k)_k$ and $(y_k)_k$, and a function $w\in L^2$ of Mass exactly $M_{max}$ such that, up to a subsequence,
$g_{(N_k,y_k)}^{-1}(u_k(t_k))\to w$ in $L^2$, where $g$ is defined in \eqref{DefG}.
\end{lemma}

As a corollary of the previous lemma, we are able to extract the minimal-Mass blow-up solution. The existence of minimal Mass blow-up solutions for the Mass-critical second order Schr\"odinger equation
was first obtained by Keraani \cite{KerNew} and analogues of Theorem \ref{3ScenariosThm} in the second order case is due to B\'egout and Vargas \cite{BegVar}, Keraani \cite{KerNew} and Killip, Tao and Visan \cite{KilTaoVis}, based on the profile decompositions in Carles and Keraani \cite{CarKer}, Bourgain \cite{Bou} and Merle and Vega \cite{MerVeg}, and in B\'egout and Vargas \cite{BegVar}.
We refer also to Kenig and Merle \cite{KenMer} and Tao, Visan and Zhang \cite{TaoVisZha}.

\begin{corollary}[Minimal-Mass blow-up solutions]\label{le:minimal-mass-blow-up-solution}
Suppose that $M_{max}<+\infty$. Then there exists $u\in S^0_{loc}(T_\ast,T^\ast)$ a maximal-lifespan strong solution of Mass equal to $M_{max}$, such that $$\Vert u\Vert_{Z(T_\ast,0)} =\Vert u\Vert_{Z(0,T^\ast)}= +\infty.$$
Furthermore, there exist two smooth functions $N: I \to \mathbb{R}_+^\ast$ and $y: I \to \mathbb{R}^n$ such that $K$ defined in \eqref{CompactnessImage}
is precompact in $L^2$.
\end{corollary}

\begin{proof}[Proof of Lemma \ref{lem4}] We proceed with similar arguments to the ones developed in Tao, Visan and Zhang \cite{TaoVisZha}, using the linear decomposition \eqref{le:linear-profile} and the stability theory in Lemma \ref{lemStab}.

Using the time-translation symmetry of \eqref{4NLS}, we may set $t_k=0$ for all $k\geq 1$.  Thus,
\begin{equation}\label{blow up in two}
\lim_{k\to +\infty} \|u_k\|_{Z(-T_k, 0)} = \lim_{k\to +\infty} \|u_k\|_{Z(0,T^k)} = +\infty.
\end{equation}
Applying Lemma~\ref{le:linear-profile} to the sequence $u_k(0)$ (which is bounded in $L^2(\R^n)$) and passing to a subsequence if necessary, we obtain a decomposition
$$
e^{it\Delta^2}u_k(0) = \sum_{1\le j\le J} \tau_{(h_k^j,t_k^j,y_k^j)}e^{it\Delta^2}\phi^j+ e^{it\Delta^2}w_k^J
$$
with the prescribed properties. By passing to subsequences for each $j$ and by using a diagonalisation argument, we may assume that, for each $j$,
$$\lim_{k\to\infty}(h_k^j)^4t_k^j=T^j\in [-\infty,\infty].$$ If $T^j$ is finite, we may refine the profile and assume $T^j=0$ and $t_k^j\equiv 0$ for all $k$. Otherwise, there holds that $\lim_{k\to +\infty} (h_k^j)^4t_k^j=\pm \infty.$ For every $j$, we define the nonlinear profiles, $v^j: I^j\times \R^n\to\C$, as follows.
\begin{itemize}
\item If $t_k^j\equiv 0$, we define $v^j$ to be the maximal lifespan solution of \eqref{4NLS} with initial data $v^j(0)=\phi^j$.

\item If $(h_k^j)^4t_k^j\to +\infty$, we define $v^j$ to be the maximal-lifespan solution of \eqref{4NLS} which scatters forward in time to $e^{it\Delta^2} \phi^j$ in the sense that
\begin{equation}\label{Scat2}
\Vert v^j(t)-e^{it\Delta^2} \phi^j\Vert_{L^2}\to 0
\end{equation}
as $t\to +\infty$.

\item If $(h_k^j)^4t_k^j\to -\infty$, we define $v^j$ to be the maximal-lifespan solution of \eqref{4NLS} which scatters backward in time to $e^{it\Delta^2} \phi^j$ in the sense that \eqref{Scat2} holds as $t\to -\infty$.
\end{itemize} Then for each $j,k\ge 1$, we define $v_k^j: I_k^j\times \R^n\to\C$ by
$$v_k^j :=\tau_{(h_k^j,t_k^j,y_k^j)}v^j.$$
where $I_k^j:= \{t: (h_k^j)^{4}(t-t_k^j)\in I^j\}$. Each $v_k^j$ is a nonlinear solution to \eqref{4NLS} with initial data $g_{(h_k^j,y_k^j)}v^j(-(h_k^j)^4t_k^j)$.
By \eqref{eq:L2-almost-ortho}, we have
\begin{equation}\label{mass decoupling}
\sum_{j=1}^{\infty} M(\phi^j)\leq  M(u_k(0))\leq M_{max}.
\end{equation}

{\bf Case I.}  Assume
$$
\sup_{j\geq 1} M(\phi^j)<M_{max}-\eps
$$ for some $\eps >0$. Then $v_k^j$ are defined globally and obey $M(v_k^j)\le M_{max}-\eps$ and, since $\Lambda$ is sublinear on $[0,M_{max}-\varepsilon]$, $$\|v_k^j\|_{Z(\R)}^\frac{2(n+4}{n}\lesssim \Lambda(M(\phi^))\lesssim_\varepsilon M(\phi^j).$$ We define an approximate solution
$$u_k^J=\sum_{j=1}^J v_k^j(t)+e^{it\Delta^2} w_k^J.$$
First since the parameters $(h_k^j,x_k^j,t_k^j)$ are orthogonal, and by \eqref{eq:err},
\begin{equation*}
\begin{split}
\lim_{J\to +\infty}\lim_{k\to +\infty} \|u_k^J\|^\frac{2(n+4)}{n}_{Z} &\le \lim_{J\to +\infty}\lim_{k\to +\infty} \|\sum_{j=1}^J v_k^j\|^\frac{2(n+4)}{n}_{Z}\le \lim_{J\to +\infty}\lim_{k\to +\infty}\sum_{j=1}^J\| v_k^j\|^\frac{2(n+4)}{n}_{Z}\\
&\le \sum_{j\to +\infty} M(\phi^j)<+\infty.
\end{split}
\end{equation*}
Secondly, the smallness at the initial time holds by definition since
$$\lim_{k\to +\infty} \|u_k^J(0)-u_k(0)\|_{L^2}=0.$$
Lastly, the smallness of the error term holds by \eqref{eq:strichartz-ortho} and \eqref{eq:err},
$$\lim_{J\to +\infty} \lim_{k\to +\infty} \|(i\partial_t+\Delta^2)u_k^J-F(u_k^J)\|_{N}=0.
$$
Now we apply the stability lemma \ref{lemStab} to conclude that $\|u_k\|_{Z}<\infty$, which leads to a contradiction to \eqref{blow up in two}.

{\bf Case II.}  Assume
$$
\sup_{j\geq 1} M(\phi^j) =M_{max}.
$$
In view of \eqref{eq:L2-almost-ortho}, there is only one profile. Thus the decomposition is reduced to
$$e^{it\Delta^2} u_k(0)=g_{(h_k,y_k)} e^{ih_k^4(t-t_k)\Delta^2} \phi+ w_k,$$
where either $t_k\equiv 0$ or $h_k^4t_k\to \pm \infty.$
We then get that
$$M(g_{(h_k,y_k)}^{-1}u_k(0)-e^{-ih_k^4t_k\Delta}\phi)\to 0,$$
and if $t_k\equiv 0$, then $u_k(0)$ converges to $\phi$ in $L^2$ and Lemma \ref{lem4} is proven. If $h_k^4t_k\to\pm \infty$, we may assume that $h_k^4t_k\to -\infty$. By time translation and monotone convergence,
$$\lim_{k\to +\infty} \|e^{i(t-h_k^4t_k)\Delta^2} \phi\|_{Z(0,\infty)}=0.$$
This yields
$$\lim_{k\to +\infty} \|g_{(h_k,y_k)}e^{ih_k^4(t-t_k)\Delta^2} \phi\|_{Z(0,\infty)}=0,$$
which, in turn, gives
$$\lim_{k\to +\infty} \|e^{it\Delta^2}u_k(0)\|_{Z(0,\infty)}=0.$$
By Lemma \ref{lemStab} again, $\lim_{k\to +\infty} \|u_k\|_{Z(0,\infty)}=0,$ a contradiction to \eqref{blow up in two}. Hence the proof of Lemma \ref{lem4} is complete.
\end{proof}

We will sketch the proof of Corollary \ref{le:minimal-mass-blow-up-solution}.
\begin{proof}[Proof of Corollary \ref{le:minimal-mass-blow-up-solution}]
By the definition of $M_{max}$, we can find a sequence $u_k: I_k \times \R^n \to \C$ of solutions with $M(u_k) \leq M_{max}$ and $\lim_{k \to +\infty} \|u_k\|_{Z(I_k)} = +\infty$.  Without loss of generality, we may take all $u_k$ to have maximal lifespan. Let $t_k\in I_k$ be such that $\|u_k\|_{Z(t_k,\sup I_k)}= \|u_k\|_{Z(\inf I_k t_k)}$.  Then,
\begin{equation}\label{eq:hyp-1}
\lim_{k\to +\infty} \|u_k\|_{Z(t_k,\sup I_k)}= \|u_k\|_{Z(\inf I_k t_k)}=\infty.
\end{equation}
Using the time-translation symmetry, we may take all $t_k=0$. Applying Lemma \ref{lem4}, and passing to a subsequence if necessary, we can locate $u_0 \in L^2$ such that $u_k(0)$ converge in $L^2$ modulo symmetries to $u_0$, that is, there exist group elements $g_k \in G$ such that $g_k u_k(0)$ converge strongly in $L^2$ to
$u_0$. We may take the $g_k$ to all be the identity, and thus $u_k(0)$ converges strongly in $L^2$ to $u_0$. In particular this implies $M(u_0) \leq M_{max}$.

Let $u: I \times \R^n \to \C$ be the maximal-lifespan solution with initial data $u(0) = u_0$ given by the local theory. We claim that $u$ blows up both forward and backward in time. Indeed, if $u$ fails to blow up forward in time, then by local theory, we have $[0,+\infty) \subset I$ and $\|u\|_{Z(0,\infty)} < \infty$.  By the stability lemma, this implies that, for sufficiently large $k$ that $[0,+\infty) \subset I_k$ and
$$ \limsup_{k \to \infty} \| u_k\|_{Z(0,+\infty)} < + \infty,$$ contradicting \eqref{eq:hyp-1}. Likewise for the case where $u$ blows up backward in time.  By the definition of $M_{max}$ this forces $M(u_0) \ge M_{max}$, and hence $M(u_0)$ must be exactly $M_{max}$.

It remains to show that a solution $u$ which blows up both forward and backward in time is precompact in $L^2$ modulo symmetries. Consider an arbitrary sequence $G u(t'_k)$ in $\{ G u(t): t \in I \}$, where $G$ denotes the symmetry group generated by spatial translation and scaling. Now, since $u$ blows up both forward and backward in time, we have $$ \|u\|_{Z(t'_k,\sup I)} = \|u\|_{Z(\inf I,t'_k)}= +\infty.$$
Applying Lemma \ref{lem4}, we see that $G u(t'_k)$ does have a convergent sequence in $L^2$.  Thus, the orbit $\{ G u(t): t \in I \}$ is precompact in $L^2$. Finding a section $g$ such that \eqref{CompactnessImage} holds is now easy.
\end{proof}

\subsection{Reduction to three scenarios.}
In this subsection, we finish the proof of Theorem \ref{3ScenariosThm} and provide more refined characterizations of the minimal-Mass blow-up solutions in Corollary \ref{le:minimal-mass-blow-up-solution} by analyzing the frequency parameter $N(t)$.

We first collect some facts on the scaling parameter $N(t)$ and the spatial translation function $y(t)$ by following the arguments developed by Killip, Tao and Visan \cite{KilTaoVis}. For $\delta>0$, and $t$, we define the interval $J_\delta(t)$ by
\begin{equation}\label{DefJDeltT}
 J_\delta(t)=(t-\delta N(t)^{-4},t+\delta N(t)^{-4}).
\end{equation}
\begin{lemma}[Local constancy of $N(t),y(t)$]\label{lem5} Let $u\in S_{loc}(I)$ be a nonzero maximal lifespan strong solution of \eqref{4NLS}, and $g(t)$ be such that $K$ as in \eqref{CompactnessImage} is precompact. Then there exists $\delta>0$ such that for any time $t_0\in I$, the following holds true:
\begin{itemize}
\item For $J_\delta(t_0)$ as in \eqref{DefJDeltT}, $J_\delta(t_0)\subset I$, and
 \begin{equation}\label{JHasNonntrivialZnorm}
\Vert u\Vert_{Z(J_\delta(t_0))}\simeq_u 1.
\end{equation}
\item $N(t)\simeq_u N(t_0)$, $\vert N(t_0)(y(t_0)-y(t))\vert\lesssim_u 1$ for any $t\in J_\delta(t_0)$.
\end{itemize}
In particular, if $N$ is bounded, then $I=\mathbb{R}$, and if $u$ blows up at time $T$, then $N(t)\ge \delta^{\frac{1}{4}}(T-t)^{-\frac{1}{4}}$.
\end{lemma}

\begin{proof}[Proof of Lemma \ref{lem5}] Since $K$ is precompact, using Strichartz estimates and local well-posedness theory,
we see that there exists $\delta>0$ such that for any $w\in K$,  the maximal nonlinear solution $W$ of \eqref{4NLS} with initial data $w$ is defined on $(-2\delta,2\delta)$ and satisfies
$ \Vert W\Vert_{Z(-2\delta,2\delta)}\lesssim 1.$
By rescaling, this gives the first claim concerning $J_\delta(t_0)$, and one inequality in \eqref{JHasNonntrivialZnorm}. On the other hand, compactness of $K$ implies that there exists $\eta$ such that for any $t$,
\begin{equation}\label{JHasNonntrivialZnormEqt2}
 \eta\le \Vert u\Vert_{Z(J_\delta(t))},
\end{equation}
which gives the second inequality in \eqref{JHasNonntrivialZnorm}. Now we prove the second claim by contradiction. If it fails for all $\delta>0$, by the definition of $J_\delta$ one can find $t_k,t'_k$ such that $s_k:=(t'_k-t_k)N(t_k)^4\to 0$ and
\begin{align*}
& \lim_{k\to +\infty} \frac {N(t_k)}{N(t'_k)}+\frac {N(t'_k)}{N(t_k)} =+\infty, \text{ or } \lim_{k\to +\infty}N(t_k)\left\vert y(t_k)-y(t'_k)\right\vert = +\infty.
\end{align*}
We define the normalization of $u$ at $t_k$, $u^{[t_k]}$ as follows,
\begin{align*}u^{[t_k]}(t):&= \tau_{(N(t_k),t_k,y(t_k))}^{-1}u(t)\\
&=N(t_k)^{-n/2}u\bigl((N(t_k)^{-4}t+t_k,N(t_k)^{-1}x+y(t_k)\bigr),\end{align*}
where $\tau$ is defined in \eqref{DefTau} and $t\in I^k:=\{s: N(t_k)^{-4}t+t_k\in I\}$. It is clear that $0\in I^k$; also note that
$$u^{[t_k]}(0)=g^{-1}_{(N(t_k),y(t_k))}[u(t_k)].$$ Hence by passing to a subsequence if necessary, from the assumption of precompactness, we can find $u_0\in L^2$ such that $u^{[t_k]}(0)$ converges strongly to a nonzero $u_0$ in $L^2$. From the local theory, let $u$ be the maximal lifespan solution with initial data $u_0$ with lifespan $I\ni 0$. By the stability lemma \ref{lemStab}, $u^{[t_k]}$ converges uniformly to $u$ on every compact interval $J\subset I$ in $C_t^0L^2_x\cap Z(J\times \R^n)$. Note that the corresponding frequency parameter and spatial centers for $u^{[t_k]}$ are in form of
$$N^{[t_k]}(s):=\frac {N(t_k+sN(t_k)^{-4})}{N(t_k)}, \quad y^{[t_k]}(s):=N(t_k)\bigl(y(t_k+sN(t_k)^{-4})-y(t_k)\bigr).$$
Hence $N^{[t_k]}(s_k)$ converges either to zero or infinity,  and $y^{[t_k]}(s)$ converges to infinity. Each possibility yields that $u^{[t_k]}(s_k)$ weakly converges to zero.

On the other hand, since $s_k$ converges to zero and $u^{[t_k]}$ converges to $u$ in $C_t^0L^2_x\cap Z(J\times \R^n)$, we see that $u^{[t_k]}(s_k)$ converges to $u(0)$ in $L^2$. Then $u(0)=0$ and hence $u$ is zero, which leads to a contradiction. Therefore the proof of Lemma \ref{lem5} is complete.
\end{proof}

\begin{corollary}[Smoothness of $g(t)=g_{(N(t),y(t))}$]\label{SmoothNy}
A consequence of the previous lemma is that the symmetry element $g$ as given by Theorem \ref{3ScenariosThm} is smooth and satisfies \eqref{VariationOfH}.
\end{corollary}
\begin{proof}
Let $\nu<\delta$, where $\delta$ is given in Lemma \ref{lem5}, $\nu$ sufficiently small so that
$$J_\nu(t_1)\subset J_\delta(t_0)\hskip.1cm\hbox{for all}\hskip.1cm t_1,t_0\hskip.1cm \hbox{such that}\hskip.1cm J_\nu(t_0)\cap J_\nu(t_1)\ne\emptyset.$$

Let $(t_k)_k$ be an increasing sequence of times such that the $J_\nu(t_k)$ cover the maximal interval of
existence of $u$ and such that the $J_{\nu/2}(t_k)$ are disjoint.
Define two intermediate rescaling functions, $\tilde{N}$ and $\tilde{y}$ as follows:
if $t_k\le t\le t_{k+1}$, then
\begin{equation}\label{DefinitionOfNewRescalingFunctions}
\begin{split}
&\tilde{N}(t)=N(t_k)+(t-t_k)\frac{N(t_{k+1})-N(t_k)}{t_{k+1}-t_k},\hskip.1cm\hbox{and}\hskip.1cm\\
&\tilde{y}(t)=y(t_k)+(t-t_k)\frac{y(t_{k+1})-y(t_k)}{t_{k+1}-t_k}.\\
\end{split}
\end{equation}
$\tilde{N}$ and $\tilde{y}$ are defined for all time $t\in I$ and are affine on each interval $[t_k,t_{k+1}]$. In particular, we get using Lemma $3.5$ that
\begin{equation}\label{NtildeAdmissible}
\begin{split}
\vert\ln\frac{\tilde{N}(t)}{N(t)}\vert
&\le \left\vert\ln\left(2\max\left(\frac{N(t_{k+1})}{N(t)},\frac{N(t_k)}{N(t)}\right)\right)\right\vert\lesssim_u1
\end{split}
\end{equation}
uniformly in $k$. Similarly, since $t\in J_\delta(t_k)$ and $J_\delta(t_{k+1})$, still by Lemma \ref{lem5}, we get that
\begin{equation}\label{YtildeAdmissible}
\begin{split}
\left\vert N(t)(y(t)-\tilde{y}(t))\right\vert
&\lesssim_u1
\end{split}
\end{equation} uniformly in $k$. Hence, we get that
$g_{(\tilde{N},\tilde{y})}^{-1}g_{(N,y)}$ remain in a compact set, and hence we can replace $N$ and $y$ by $\tilde{N}$ and $\tilde{y}$, in the sense that the set
$$\tilde{K}=\{g_{(\tilde{N}(t),\tilde{y}(t))}^{-1}u(t),t\in I\}$$
is still precompact in $L^2$. Now, $\tilde{y}$ and $\tilde{N}$ are piecewise affine functions. In particular, since $t_{k+1}-t_k\ge \nu/2N(t_k)^{-4}$, we see that, for $t_k<t<t_{k+1}$,
\begin{equation}\label{DerAffine1}
\begin{split}
\vert \tilde{N}^{-5}(t)\partial_t\tilde{N}(t)\vert&=\left\vert \tilde{N}^{-5}(t)\frac{N(t_{k+1})-N(t_k)}{t_{k+1}-t_k}\right\vert\\
&\le \frac{2}{\nu}\left(\frac{\tilde{N}(t)}{N(t)}\frac{N(t)}{N(t_k)}\right)^{-5}\left\vert\frac{N(t_{k+1})-N(t_k)}{N(t_k)}\right\vert
\lesssim_u 1,
\end{split}
\end{equation}
uniformly in $k$. Similarly, we get
\begin{equation}\label{DerAffine2}
\begin{split}
\left\vert \tilde{N}^{-3}(t)\partial_t\tilde{y}(t)\right\vert &=\left\vert \tilde{N}^{-3}(t)\frac{y(t_{k+1})-y(t_k)}{t_{k+1}-t_k}\right\vert\\
&\le \frac{2}{\nu}\left(\frac{N(t_k)}{N(t)}\right)^{-3}N(t_k)\left\vert y(t_{k+1})-y(t_k)\right\vert\lesssim_u1
\end{split}
\end{equation}
uniformly in $k$. Hence, \eqref{VariationOfH} hold for $\tilde{N}$ and $\tilde{y}$ except at $t_k$. Then it suffices to smooth out $\tilde{N}$ and $\tilde{y}$ to get smooth functions. To do this, let $\chi$ be a smooth nonnegative bump function supported in $[-1,1]$ and satisfying $\int\chi=1$, and $\chi_k=\mu\chi(\mu x)$ for $10/\mu=\nu N(t_k)^{-4}$. Then on each interval $J_\nu(t_k)$, we replace $\tilde{N}$ and $\tilde{y}$ by $N_k=\tilde{N}\ast\chi_k$ and $y_k=\tilde{y}\ast\chi_k$. This only modifies $\tilde{N}$ and $\tilde{y}$ in $(t_k-\mu,t_k+\mu)$, so that defining $N_u$ and $y_u$ by $N_u(t)=N_k(t)$ and $y_u(t)=y_k(t)$ if $t\in J_\nu(t_k)$, we get global smooth functions as in Corollary \ref{SmoothNy}.
\end{proof}

By a similar argument as in Killip, Tao and Visan \cite{KilTaoVis}, see also Killip and Visan \cite{KilVis2}, we have the following control on the scattering size via the frequency parameter, $N(t)$.
\begin{lemma}[Control of Strichartz norms via $N(t)$]\label{le:Strichartz-Via-N} Let $u$ be a nonzero solution to \eqref{4NLS} with lifespan $I$, which is precompact with parameters $N(t),y(t)$. Then for any compact interval $J\subset I$,
\begin{equation}\label{BoundOnNonlinearTermByH}
 \int_{J}N(s)^{4}ds\lesssim_u \Vert u\Vert_{Z(J)}^\frac{2(n+4)}{n}\lesssim_u 1+\int_{J}N(s)^{4}ds.
\end{equation}
By Strichartz estimates and conservation of Mass, this implies that for all S-admissible pairs $(p,q)$,
$$\Vert \vert\nabla\vert^\frac{2}{p}u\Vert_{L^p(I,L^q)}^p\lesssim_u 1+\int_IN(t)^{4}dt.$$
\end{lemma}

We also need the following lemma.
\begin{lemma}\label{lem6} Let $u_k$ be a sequence of maximal lifespan solutions defined on $I_k$ satisfying that
$$\tilde{K}=\{g_k^{-1}(t)u_k(t)=g_{(N_k(t),y_k(t))}^{-1}u_k(t):t\in I_k,k\ge 0\}$$
is precompact in $L^2$ for some rescaling functions $g_k$, and suppose that
$$u_k\to W\hskip.1cm\hbox{in}\hskip.1cm C^0_{loc}(I,L^2)\cap L^\frac{2(n+4)}{n}(I,L^\frac{2(n+4)}{n}),$$ where
$W\in S(I)$ is a solution to \eqref{4NLS}
such that $M(W)=M_{max}$ and such that $\{G(t)^{-1}W(t),t\in I\}$ is precompact in $L^2$ for some rescaling function $G(t)=g_{(N(t),Y(t))}$. Then for all $t \in I$,
\begin{equation}\label{ContinuityLemma}
\begin{split}
&0<\liminf_{k\to \infty} N_k(t)\lesssim_u N(t)\lesssim_u \limsup_{k\to +\infty} N_k(t)<+\infty,\hskip.1cm\hbox{and}\\
&\limsup_{k\to \infty} N(t)\left\vert y_k(t)-Y(t)\right\vert\lesssim_u 1.
\end{split}
\end{equation}
\end{lemma}
\begin{proof}[Proof of Lemma \ref{lem6}]  Let $(u_k)_k$ be a sequence of maximal lifespan solutions as above converging to $W\in S(I)$ with rescaling function $G$. Suppose that there exists a sequence of time $s_p\in I$, a subsequence $k'=k'(p)$ such that
\begin{equation}\label{AsymptoticOrthogonality}
\frac {N(s_p)}{N_{k'}(s_p)}+\frac {N_{k'}(s_p)}{N(s_p)}+N(s_p)\left\vert y_{k^\prime}(s_p)-Y(s_p)\right\vert>p.
\end{equation}
By assumption, $u_k\to W$ in $C(J,L^2)$ for any compact interval $s_p\in J\subset I$. Hence we can find an element $k\ge p$ such that $M(u_k(s_p)-W(s_p))\le M_{max}/8$ and that \eqref{AsymptoticOrthogonality} holds true with $k$ instead of $k^\prime$. This defines a sequence $k(p)$ such that for any $p$,
$$u_{k(p)}(s_p)=W(s_p)+r_p$$
with $M(r_p)\le M_{max}/8$. By passing to a subsequence if necessary, we may assume that $G(s_p)^{-1}W(s_p)\to \bar{w}$ in $L^2$. In particular $M(\bar{w})=M_{max}$. However
$$G(s_p)^{-1}W(s_p)=\left[G(s_p)^{-1}g_{k(p)}(s_p)\right]g_{k(p)}(s_p)^{-1}u_{k(p)}(s_p)-G(s_p)r_p.$$
Now by assumption, $g_{k(p)}(s_p)^{-1}u_{k(p)}(s_p)$ remains in a compact subset, and we have that $G(s_p)^{-1}g_{k(p)}(s_p)$ converges weakly to $0$.
As a consequence any weak limit $w^\prime$ of $G(s_p)^{-1}W(s_p)$ satisfies $M(w^\prime)\le M( r_p)\le M_{max}/8$. This contradicts $M(\bar{w})=M_{max}$. This finishes the proof of \eqref{ContinuityLemma} and of the lemma.
\end{proof}
Now we present a combinatorial argument to prove Theorem \ref{3ScenariosThm}. We refer to Killip, Tao and Visan \cite{KilTaoVis} for the original treatment in the context of the $L^2$-critical Schr\"odinger equation.
\begin{proof}[Proof of the second part of Theorem \ref{3ScenariosThm}] Let $u$ be a maximal lifespan solution given as in Lemma \ref{le:minimal-mass-blow-up-solution}; assume $N(t)$ and $y(t)$ are the corresponding frequency parameter function and the spatial translation function, respectively.  Then we define the oscillating function
\begin{equation}\label{DefOfOscillFunc}
 Osc(\kappa)=\inf_{t_0\in I}\frac{\sup_{t\in I\cap  J_\kappa(t_0)}N(t)}{\inf_{t\in I\cap J_\kappa(t_0)}N(t)},
\end{equation}
where $J_\kappa(t_0)$ is defined in \eqref{DefJDeltT}; it is an increasing function of $\kappa$. \\

{\bf Case I.} $\lim_{\kappa\to +\infty} Osc(\kappa)<\infty$. Then there exist a sequence $t_k\in I$, a sequence $\kappa_k\to +\infty$, and $A>0$ such that
\begin{equation*}
 \frac{\sup_{t\in I\cap J_{\kappa_k}(t_k)}N(t)}{\inf_{t\in I\cap J_{\kappa_k}(t_k)}N(t)}<A
\end{equation*}
for every $k$. Since $N(t)$ is bounded on $J_{\kappa_k}(t_k)$, this implies that $J_{\kappa_k}(t_k)\subset I$. Furthermore by definition of $Osc$, for any $t\in J_{\kappa_k}(t_k)$, $N(t)\simeq_AN(t_k)$. By Lemma \ref{le:Strichartz-Via-N}, we have
$$\left\Vert u\right\Vert_{Z(J)}^{\frac{2(n+4)}{n}} \lesssim_{A,u} 1+N(t_k)^{4}\vert J\vert$$
for all  $J\subset J_{\kappa_k}(t_k)$. By compactness of $K$, after passing to a subsequence if necessary, we may assume that $g(t_k)^{-1}u(t_k)$ converges in $L^2$ to an element $w$. Let $W$ be the nonlinear solution with initial data $W(0)=w$ with maximal interval $I^\prime$.
By the stability Lemma \ref{lemStab}, we see that
$$u_k(t)=g_{\left(N(t_k),y(t_k)\right)}^{-1}[u\bigl(N(t_k)^{-4}t+t_k\bigr)] \to W\hskip.2cm\hbox{in}\hskip.2cm C_{loc}(I',L^2)\cap L^\frac{2(n+4)}{n}(I^\prime,L^\frac{2(n+4)}{n})
\hskip.1cm ,$$
and that for any $\kappa$,
$$\left\Vert W\right\Vert_{Z(-\kappa,\kappa)}^{\frac{2(n+4)}{n}}\lesssim_{A,u} 1+\vert \kappa\vert
\hskip.1cm .$$
Hence by local theory $I^\prime=\mathbb{R}$. Furthermore $W$ satisfies the conclusions in Lemma \ref{le:minimal-mass-blow-up-solution}. Let $G=g_{(N_W,Y)}$ be the rescaling function corresponding to $W$, and let also $g_k(s)=g(t_k)^{-1}g(t_k+N(t_k)^{-4}s)$. Then, applying Lemma \ref{lem6} with $g_k$, $u_k$, $G$ and $W$, we get that $N_W(t)\simeq 1$, and hence, we can assume that $\forall t\in\mathbb{R}, N_W(t)=1$, and we get the last scenario in Theorem \ref{3ScenariosThm}.

{\bf Case II.} $\lim_{\kappa\to +\infty}Osc(\kappa)=+\infty$. We introduce the quantity
\begin{equation}\label{DefOfAFunction}
a\left(t_0\right)=\frac{\inf_{t\in I,t\le t_0}N(t)+\inf_{t\in I,t\ge t_0}N(t)}{N(t_0)}
\end{equation}
where  $t_0\in I$ is arbitrary. We consider two cases.\\

{\bf Subcase 1.} $\inf_{t_0\in I}a(t_0)=0$. In this case, there exists a sequence of times $t_k\in I$ such that $a(t_k)\to 0$. Besides, by definition, we can find $t^-_k<t_k,t_k^\prime<t^+_k$ satisfying
\begin{equation}\label{CondOnTkInCaseInfA=0}
\begin{split}
&N(t_k^\prime)=\sup_{t\in [-t^-_k,t^+_k]}N(t)\\
&\frac{N(t^-_k)}{N(t^\prime_k)}\to 0\hskip.1cm,\hskip.1cm\hbox{and}\hskip.1cm\frac{N(t^+_k)}{N(t^\prime_k)}\to 0,
\end{split}
\end{equation}
as $k\to +\infty$. Then, by passing to a subsequence, we may assume that $g(t^\prime_k)^{-1}u(t^\prime_k)\to w$ in $L^2$ for some $w\in L^2$ satisfying $M(w)=M_{max}$, and such that if $W$ is the maximal solution to \eqref{4NLS} with initial data $W(0)=w$ as in Lemma \ref{le:minimal-mass-blow-up-solution}. Let $s^\pm_k=N(t^\prime_k)^{4}\left(t^\pm_k-t^\prime_k\right)$.
Applying \eqref{BoundOnNonlinearTermByH} to the normalized $$u_k(t):=g_{(N(t^\prime_k),y(t^\prime_k))}^{-1}[u(N(t'_k)^{-4}t+t_k)],$$ we obtain that
$$\Vert u_k\Vert_{Z(s_k^-,0)}^\frac{2(n+4)}{n}=\Vert u\Vert_{Z(t_k^-,t_k^\prime)}^\frac{2(n+4)}{n}\lesssim 1+\vert s_k^-\vert.$$
We aim to prove that $s^{-}_k$ is unbounded by contradiction. Suppose that $s^-_k$ remains bounded; by passing to a subsequence if necessary, one may assume that $s^-_k\to S^-$. Then, since
$$u_k \to W\hskip.2cm\hbox{in}\hskip.2cm L^\frac{2(n+4)}{n}(I,L^\frac{2(n+4)}{n}),$$
we have that, for any compact interval $J\subset ]S^-,0]$,
$$\left\Vert W\right\Vert_{Z(J)}^{\frac{2(n+4)}{n}} \lesssim 1+\vert S^-\vert .$$
This yields that $W$ is well-defined on $[S^-,0]$, and $u_k$ converges to $W$ in $C([S^--\epsilon,0],L^2)$ for some $\epsilon>0$. In particular $u_k(s_k^-)\to W(S^-)$ in $L^2$. However,
$$u_k(s_k^-)=g_{\left(N(t_k^-)/N(t_k^\prime),N(t_k^\prime)\left(y(t_k^-)-y(t_k^\prime)\right)\right)}\left(g(t_k^-)^{-1}u(t_k^-)\right),$$
and by passing to a subsequence, we may assume that $g(t^-_k)^{-1}u(t^-_k)$ converges in $L^2$. Then, since $N(t^-_k)/N(t^\prime_k)\to 0$, we see that $u_k(s_k^-)$ has to converge weakly to $0$. But this contradicts $M\left(W(S^-)\right)=M\left(W(0)\right)=M_{max}$. Hence $s^-_k$ is unbounded; by passing to a subsequence, we may assume that $s^-_k\to -\infty$. Then, using \eqref{BoundOnNonlinearTermByH}, we see that for any compact time interval $J=[S,T]\subset (-\infty,0]$, we have
\begin{equation*}
\begin{split}
\Vert W\Vert_{Z(J)}&\le\limsup_k\Vert u_k\Vert_{Z(J)}=\limsup_k\Vert u\Vert_{Z([t_k^\prime+SN(t^\prime_k)^{-4},t_k^\prime+TN(t^\prime_k)^{-4}])}\\
&\lesssim \left(1+T-S\right)^\frac{n}{2(n+4)}
\end{split}
\end{equation*}
hence, $W$ is defined on $(-\infty,0]$; by a similar discussion on $s^+_k$, $W$ is also well-defined on $[0,\infty)$. Let $G(t)=g_{(N_W(t),Y_W(t))}$ be the rescaling function associated to $W$ on $I=\R$. By the way we define $N(t'_k)$ and Lemma \ref{lem6}, $$\sup_{t\in I} N_W(t)\lesssim 1.$$
Now in order to classify $W$ into the double high-to-low cascade, it remain to show that
$$\lim_{t\to -\infty} N_W(t)=\lim_{t\to +\infty} N_W(t)=0.$$
We establish the first limit by contradiction. The second limit follows similarly. We suppose first that $\liminf_{t\to-\infty}N_W(t)>0$, then we may assume that $$N_W(t)\simeq_u 1.$$
Feeding this information back to $u_k$ and then to $u$, it contradicts the assumption that $\lim_{\kappa}Osc(\kappa)=\infty$. Hence the double high-to-low cascade scenario is established.

{\bf Subcase 2.} ~$a=\inf_{t_0\in I}a(t_0)>0$.
We pick $\epsilon<a/2$, and define
\begin{equation}\label{DefOfJIntervals}
\begin{split}
&\mathcal{J}^-=\{t\in I:\forall s\le t, N(s)\ge \epsilon N(t)\}\\
&\mathcal{J}^+=\{t\in I:\forall s\ge t, N(s)\ge \epsilon N(t)\}.
\end{split}
\end{equation}
By hypothesis, $I=(-T_\ast,T^\ast)=\mathcal{J}^-\cup\mathcal{J}^+$. By time reversal symmetry, we can suppose that $\mathcal{J}^+$ is not empty. Let $t_0\in \mathcal{J}^+$. We claim that any sufficiently late time $t\ge t_0$ belongs to $\mathcal{J}^+$. Suppose it were not, then there exists a sequence $(t_k)_k$ such that $t_k\in\mathcal{J}^-$, $t_k\ge t_0$ and $t_k\to T^\ast$. We have that $N(t_0)\ge\epsilon N(t_k)\ge\epsilon^{2}N(t_0)$, and since any time $t\in [t_0,t_k]$ belongs to $\mathcal{J}^-$ or to $\mathcal{J}^+$, we have, letting $k\to +\infty$ that $\forall t\in (t_0,T^\ast)$, $N(t)\simeq_\epsilon N(t_0)$. Then by Lemma \ref{lem5}, we have that $T^\ast=+\infty$, which yields that $Osc(\kappa)$ remains bounded. This is a contradiction.

Hence we may suppose that $[t_0,T^\ast)\subset\mathcal{J}^+$ for some $t_0\in I$. In this case, we will construct a sequence of times $(t_k)_{k\ge 0}$ recursively, starting from $t_0$ above, as follows.

Let $B$ be such that $Osc(B)> 2\epsilon^{-1}$ and define $s_k=t_k+B\epsilon^{-4}N(t_k)^{-4}$. By definition, $N(s_k)\ge\epsilon N(t_k)$, which yields that $J_B(s_k)\subset [t_k,t_k+2B\epsilon^{-4}N(t_k)^{-4}]$. Then by the definition of $B$ and the fact that the times after $t_k$ are in $\mathcal{J}^+$,
\begin{align*}&\sup_{t\in [t_k,t_k+2B\epsilon^{-4}N(t_k)^{-4}]}N(t)\ge \sup_{t\in J_B(s_k)}N(t)\\
&\quad \ge Osc(B)\inf\{N(t): t\in I\cap J_B(s_k)\}\ge 2\eps^{-1}\eps N(t_k)=2N(t_k).\end{align*}
Hence we select $t_{k+1}$ such that
\begin{equation}\label{CondOntk}
\begin{split}
&t_k\le t_{k+1} \le \min\left(t_k +2B\epsilon^{-4}h(t_k)^4,T^\ast\right),\\
&2N(t_k)\le N(t_{k+1})\lesssim_u N(t_k),\\
&\forall t\in [t_k,t_{k+1}],  N(t)\simeq_u 2N(t_k).
\end{split}
\end{equation}
By \eqref{CondOntk}, we see that $$N(t_k)\ge 2^{k}N(t_0),\quad 0<t_{k+1}-t_k\lesssim_u 2^{-4k}N(t_0)^{-4}.$$ Hence $t_k$ converges to a limit $t_\infty\le T^\ast$ and for all $k$, $t_\infty-t_k\lesssim 2^{-4k}$. Furthermore $N(t)\to +\infty$ as $t\to t_\infty$; this yields that $t_\infty=T^\ast$ and $T^\ast-t_k\lesssim N(t_k)^{-4}$. Combining this with Lemma \ref{lem5}, we get that $T^\ast-t_k\simeq N(t_k)^{-4}$. Using again \eqref{CondOntk}, we get that $T^\ast-t\simeq N(t)^{-4}$ for all $k$ and all $t\in [t_k,t_{k+1}]$, and hence for any $t\in [t_0,T^\ast)$. Modifying $N$ by a bounded function, we may thus suppose that $$N(t)=(T^\ast-t)^{-\frac{1}{4}}, \forall ~t\in (t_0,T^*).$$
Finally we consider the normalization $u^{[t_k]}$ as in Lemma \ref{lem5} at times $t_k$, and apply Lemma \ref{le:minimal-mass-blow-up-solution} and time translation and reversal symmetry of \eqref{4NLS} as in \cite{KilTaoVis}, we obtain a self-similar blow-up solution as desired. This complete the proof of Theorem \ref{3ScenariosThm}.
\end{proof}

Finally, we remark that the analysis developed in this section, together with the information about $M_{max}$ to be obtained in the following sections allow us to give a description of the structure of a sequence of solutions of \eqref{4NLS} with bounded Mass. The loss of compactness for this sequence is only due to the invariance of the equation \eqref{DefTau}. More precisely, we have the following result.
\begin{theorem}\label{LossOfCompThm} Let $(u_k)_k$ be a sequence of strong solutions of \eqref{4NLS} such that $\sup_kM(u_k)<M_{max}$. Then there exists a sequence $(U^\alpha)_\alpha$ of global strong solutions of \eqref{4NLS}, and a family $((h^\alpha_k,t^ \alpha_k,x^\alpha_k)_k)_ \alpha$ of pairwise orthogonal scale-cores such that, up to a subsequence,
\begin{equation}\label{XXX}
u_k= \sum_ {\alpha=1}^A \tau_ {(h_k^\alpha,t_k^\alpha,x_k^\alpha)}U^ \alpha+e^{it\Delta^2}w_k^A+r_k^A
\end{equation}
for all $k$ and $A$, where $w_k^A \in L^2$ and $r_k^A\in S$ for all $k$ and $A$,
$$\lim_{A\to+\infty}\limsup_{k\to +\infty}\Vert e^{it\Delta^2}w_k^A\Vert_Z = \lim_{A\to +\infty}\limsup_{k\to +\infty}\Vert r_k^A\Vert_{Z} = 0.$$ Moreover, for any $A$, and $t$, \eqref{XXX} is asymptotically $L^2$-orthogonal, i.e.,
\begin{equation}\label{DefAsymptOrtho}
\lim_{k\to+\infty}\Vert u_k\Vert_{L^2}^2=\sum_{\alpha=1}^A \Vert U^\alpha\Vert_{L^2}^2 +\lim_{k\to+\infty}\bigl(\Vert w_k^A\Vert_{L^2}^2+\Vert r_k^A\Vert_{L^2}^2\bigr).
\end{equation} and
\begin{equation}\label{YYY}
\Vert u_k\Vert_Z^{\frac{2(n+4)}{n}}\le \sum_{\alpha=1}^{+\infty}\Vert U^\alpha\Vert_Z^{\frac{2(n+4)}{n}} + o(1),
\end{equation}
where $o(1) \to 0$ as $k \to +\infty$.
\end{theorem}
This theorem is the analogue in the case of \eqref{4NLS} of the results in Bahouri and G\'erard \cite{BahGer} about the Energy-critical wave equation, and Keraani \cite{Ker} about the Energy-critical Schr\"odinger equation. Since the proof of Theorem \ref{LossOfCompThm} is very similar to the proof of Lemma \ref{lem4}, we omit it.


\section{Gaining regularity when the linear term is absent}\label{Sec-Abs}
In this section, we establish general results that allow to gain regularity when the linear term in the Duhamel formula \eqref{DuhamelFormula} has a negligible contribution. This takes the form of results concerning three nonnegative functions, $\mathcal{M}$, $\mathcal{Z}$, $\mathcal{N}$ of a dyadic argument which are essentially non-increasing in the sense that
$$\mathcal{M}(A)\lesssim \mathcal{M}(B)$$
whenever $B\le A$ (and similarly for $\mathcal{Z}$, $\mathcal{N}$), and  which represent frequency-localized versions of the mass, scattering norm and control on the nonlinear forcing. Although the precise form of these functions, and the way to obtain the hypothesis depend on the scenarios, they can hereafter be treated in a similar way.

More precisely, in this section, we consider three functions, $\mathcal{M}$, $\mathcal{Z}$, $\mathcal{N}$ satisfying the Strichartz estimates
\begin{equation}\label{Hyp1}
\mathcal{Z}(A)\lesssim \mathcal{M}(A)+\mathcal{N}(A),
\end{equation}
the boundedness estimate
\begin{equation}\label{Hyp2}
\mathcal{M}(A)+ \mathcal{Z}(A)+\mathcal{N}(A)\lesssim_u 1,
\end{equation}
the localization estimate
\begin{equation}\label{Hyp3}
\lim_{A\to +\infty}\mathcal{M}(A)=0,
\end{equation}
together with two estimates that allow us to neglect the effect of the mass in \eqref{Hyp1}.
These are a vanishing of the effect of the linear part of the Duhamel formula on the scattering norm
\begin{equation}\label{Hyp4}
\mathcal{Z}(A)\lesssim A^{-\frac{1}{2}}+\mathcal{N}(A/2)
\end{equation}
and a condition that expresses the vanishing of the linear part in that the Mass is controlled purely by the effect of the nonlinearity:\footnote{the various additive constant like $A^{-\frac{5}{2}}$ in the following estimates correspond to saturations in the gain of derivatives and reflects the fact that our nonsmooth nonlinearity does not allow us to gain infinitely many derivatives.}
for all $\sigma>0$,
\begin{equation}\label{Hyp5}
\mathcal{N}(A)\lesssim A^{-\sigma}\Rightarrow\mathcal{M}(A)\lesssim A^{-\sigma}+A^{-\frac{5}{2}}.
\end{equation}
As their name indicate, \eqref{Hyp1}-\eqref{Hyp3} will be quite straightforward to obtain, while \eqref{Hyp4} and \eqref{Hyp5} already essentially contain the gain in regularity and are the key hypothesis.

We also note that in the situation we consider, we can add another relation.

\begin{lemma}\label{Lem1}
Suppose that $\mathcal{M}$, $\mathcal{Z}$, $\mathcal{N}$ as above are given by
\begin{equation}\label{ImportantQuantitiesAbs}
\begin{split}
\mathcal{M}(A) &=\Vert P_{\ge A}u\Vert_{L^\infty(I,L^2)}\\
\mathcal{Z}(A) &=\Vert P_{\ge A}u\Vert_{Z(I)}\\
\mathcal{N}(A) &=\Vert P_{\ge A}F(u)\Vert_{N(I)}\\
\end{split}
\end{equation}
for some interval $I\subset\mathbb{R}$ and $u$ satisfying one of the three scenarios in Theorem \ref{3ScenariosThm}.
Then, for all $A$ sufficiently large,
\begin{equation}\label{NonlinearEstimates}
\begin{split}
\mathcal{N}(A)\lesssim_u & \mathcal{Z}(A/8)^{1+\frac{8}{n}}+Z(A/8)\left(\sum_{N\le A/8}\left(\frac{N}{A}\right)^\frac{n^2}{8(n+4)}\mathcal{Z}(N)\right)^\frac{8}{n}\\
&+\left(\sum_{N\le A}\left(\frac{N}{A}\right)^\frac{n\alpha}{n+8}\mathcal{Z}(N)\right)^\frac{n+8}{n}+\left(\sum_{N\le A/8}\left(\frac{N}{A}\right)\mathcal{Z}(N)\right)^\frac{7}{3}
\end{split}
\end{equation}
where the summations are taken over all $N=2^{-j}A$ for $j\ge 0$.
\end{lemma}
\begin{proof}
To prove \eqref{NonlinearEstimates}, we split $u$ into its high and low frequency components,
$$u=u_l+u_h=P_{< A/8}u+P_{\ge A/8}u.$$
We first remark that $$\vert F(u)-F(u_l)\vert\lesssim \vert u_{h}\vert^{1+\frac{8}{n}}+\vert u_{h}\vert \vert u_{l}\vert^\frac{8}{n}.$$
Consequently,
\begin{equation}\label{NLE1}
\begin{split}
\mathcal{N}(A)&\lesssim \Vert P_{\ge A}F(u_{l})\Vert_{N(I)}+\Vert \vert u_{h}\vert^{1+\frac{8}{n}}\Vert_{N(I)}+\Vert u_{h}\vert u_{l}\vert^\frac{8}{n}\Vert_{N(I)}.\\
\end{split}
\end{equation}
Besides, by Sobolev's inequality, we remark that
\begin{equation}\label{NLE2}
\begin{split}
\Vert \vert u_{h}\vert^{1+\frac{8}{n}}\Vert_{N(I)} &\lesssim\Vert \vert P_{\ge A/8}u\vert^{1+\frac{8}{n}}\Vert_{L^\frac{2(n+4)}{n+8}(I,L^\frac{2(n+4)}{n+8})}\\
&\lesssim Z(A/8)^{1+\frac{8}{n}},
\end{split}
\end{equation}
while using Bernstein inequalities \eqref{BernSobProp}, we obtain that
\begin{equation}\label{NLE3}
\begin{split}
\Vert u_{h}\vert u_{l}\vert^\frac{8}{n}\Vert_{N(I)}
&\lesssim A^{-\frac{n}{n+4}}\Vert u_{h}\vert u_{l}\vert^\frac{8}{n}\Vert_{L^\frac{2(n+4)}{n+8}(I,L^\frac{2(n+4)}{n+6})}\\
&\lesssim A^{-\frac{n}{n+4}}
\Vert
u_{h}\Vert_{Z(I)}\Vert u_{l}\Vert_{L^\frac{2(n+4)}{n}(I,L^\frac{8(n+4)}{3n})}^\frac{8}{n}\\
&\lesssim A^{-\frac{n}{n+4}}\mathcal{Z}(A/8)\left(\sum_{N\le A/8}N^\frac{n^2}{8(n+4)}\Vert P_{N}u\Vert_{Z(I)}\right)^\frac{8}{n}\\
&\lesssim\mathcal{Z}(A/8)\left(\sum_{N\le A/8}\left(\frac{N}{A}\right)^\frac{n^2}{8(n+4)}\mathcal{Z}(N)\right)^\frac{8}{n}.
\end{split}
\end{equation}
To estimate the first term, we treat the different dimensions differently.
First, suppose that $n\ge 7$. In this case, $n\alpha/(n+8)\ge 1$ for $\alpha=2(n+5)/(n+4)$, and we let $q=2(n+4)(n+8)/(n^2+4n-20)$, $r=2n(n+4)(n+8)/(n^3+6n^2+4n+64)$, $p=n(n+4)(n+8)/(4n^2+22n-32)$, $p_1=nq/2$, $p_2=n(n+4)(n+8)/(3(n^2+6n-4))$.
We start with the following estimate which is a consequence of the work in Visan \cite[Appendix A]{Vis},
\begin{equation}\label{FractionalDerivative}
\begin{split}
\Vert \vert\nabla\vert^\frac{6}{n+4}F^\prime(u_l)\Vert_{L^p}&\lesssim \Vert u_l\Vert_{L^q}^\frac{2}{n}\Vert \vert\nabla\vert^\frac{n}{n+4}u_l\Vert_{L^\frac{6p_2}{n}}^\frac{6}{n}\\
&\lesssim \Vert \vert\nabla\vert^\frac{n\alpha}{n+8}u_l\Vert_{L^{\frac{2(n+4)}{n}}}^\frac{8}{n},
\end{split}
\end{equation}
and consequently, we get with \eqref{FractionalDerivative} that
\begin{equation*}
\begin{split}
\Vert P_{\ge A}F(u_l)\Vert_{N(I)}
&\lesssim A^{-\frac{n}{n+4}}\Vert P_{\ge A}F(u_{<A/8})\Vert_{L^\frac{2(n+4)}{n+8}(I,L^\frac{2(n+4)}{n+6})}\\
&\lesssim A^{-\frac{2n+4}{n+4}}\Vert P_{\ge A}\left(F^\prime(u_l)\nabla u_l\right)\Vert_{L^\frac{2(n+4)}{n+8}(I,L^\frac{2(n+4)}{n+6})}\\
\end{split}
\end{equation*}
Now, since $\left(P_{<A/4}F^\prime(u_l)\right)\left(P_{<A/8}\nabla u\right)$ is supported in frequency space in the ball of radius $A/2$, we discard it and get
\begin{equation*}
\begin{split}
&\Vert P_{\ge A}F(u_l)\Vert_{N(I)}\\
&\lesssim A^{-\frac{2n+4}{n+4}}\Vert \left(P_{\ge A/4}F^\prime(u_l)\right)\nabla u_l\Vert_{L^\frac{2(n+4)}{n+8}(I,L^\frac{2(n+4)}{n+6})}\\
&\lesssim A^{-\frac{2n+4}{n+4}}\Vert \nabla u_l\Vert_{L^\frac{2(n+4)}{n}(I,L^r)}\Vert P_{\ge A/4}F^\prime(u_l)\Vert_{L^\frac{2(n+4)}{8}(I,L^p)}\\
&\lesssim A^{-\frac{2n+10}{n+4}}\Vert \vert\nabla\vert^\frac{n\alpha}{n+8} u_l\Vert_{Z(I)}\Vert \vert\nabla\vert^\frac{6}{n+4}P_{\ge A/4}F^\prime(u_l)\Vert_{L^\frac{2(n+4)}{8}(I,L^p)}\\
&\lesssim A^{-\frac{2n+10}{n+4}}\Vert \vert\nabla\vert^\frac{n\alpha}{n+8} u_l\Vert_{Z(I)}^\frac{n+8}{n}\\
&\lesssim \left(\sum_{N\le A/8}\left(\frac{N}{A}\right)^\frac{n\alpha}{n+8}\mathcal{Z}(N)\right)^\frac{n+8}{n}
\end{split}
\end{equation*}
and this concludes the proof when $n\ge 7$.

In case $3\le n\le 6$, we proceed similarly except that we take two derivatives and use H\"older's inequality and Bernstein properties \eqref{BernSobProp} to get
\begin{equation*}
\begin{split}
&\Vert P_{\ge A}F(u_l)\Vert_{N(I)}\\
&\lesssim A^{-\frac{n}{n+4}}\Vert P_{\ge A}F(u_{<A/8})\Vert_{L^\frac{2(n+4)}{n+8}(I,L^\frac{2(n+4)}{n+6})}\\
&\lesssim A^{-\frac{2n+4}{n+4}}\Vert \left(P_{\ge A/4}F^\prime(u_l)\right)\nabla u_l\Vert_{L^\frac{2(n+4)}{n+8}(I,L^\frac{2(n+4)}{n+6})}\\
&\lesssim A^{-\frac{2n+4}{n+4}}\Vert \nabla u_l\Vert_{Z(I)}\Vert P_{\ge A/4}F^\prime(u_l)\Vert_{L^\frac{2(n+4)}{8}(I,L^\frac{2(n+4)}{6})}\\
&\lesssim A^{-\frac{3n+8}{n+4}}\Vert\nabla u_l\Vert_{Z(I)}\Vert F^{\prime\prime}(u_l)\nabla u_l\Vert_{L^\frac{2(n+4)}{8}(I,L^\frac{2(n+4)}{6})}.\\
\end{split}
\end{equation*}
Using again H\"older's inequality and Bernstein property, we see that\footnote{when $n=6$, we skip the line before last in \eqref{EquationForDim3n6}.}
\begin{equation}\label{EquationForDim3n6}
\begin{split}
&\lesssim A^{-\frac{3n+8}{n+4}}\Vert \nabla u_l\Vert_{Z(I)}^2\Vert u_l\Vert_{L^\frac{2(n+4)}{n}(I,L^\frac{2(n+4)(8-n)}{n(6-n)})}^\frac{8-n}{n}\\
&\lesssim A^{-\frac{3n+8}{n+4}}\Vert \nabla u_l\Vert_{Z(I)}^2\Vert \vert\nabla\vert^\frac{n^2}{(n+4)(8-n)}u_l\Vert_{Z(I)}^\frac{8-n}{n}\\
&\lesssim \left(\sum_{N\le A/8}\left(\frac{N}{A}\right)\mathcal{Z}(N)\right)^2\left(\sum_{N\le A/8}\left(\frac{N}{A}\right)^\frac{n^2}{(n+4)(8-n)}\mathcal{Z}(N)\right)^\frac{8-n}{n}.
\end{split}
\end{equation}
Now, if $n=6$, since $N/A\le 1$ in the sum above, we can replace the exponent $n^2/((n+4)(8-n))$ by $1$ in the last product and conclude the proof. If $3\le n\le 5$, $n^2/((n+4)(8-n))> n/(3(8-n))$ and then we apply H\"older's inequality and \eqref{Hyp2} to obtain
\begin{equation*}
\begin{split}
\left(\sum_{N\le A/8}\left(\frac{N}{A}\right)^\frac{n^2}{(n+4)(8-n)}\mathcal{Z}(N)\right)^\frac{8-n}{n}
&\lesssim \left(\sum_{N\le A/8}\left(\frac{N}{A}\right)\mathcal{Z}(N)^\frac{3(8-n)}{n}\right)^\frac{1}{3}\\
&\lesssim_u \left(\sum_{N\le A/8}\left(\frac{N}{A}\right)\mathcal{Z}(N)\right)^\frac{1}{3}
\end{split}
\end{equation*}
and this concludes the proof in case $3\le n\le 5$. Finally, when $n=1, 2$, we need to modify \eqref{EquationForDim3n6} by taking one more derivative. Since the treatment is very similar and we are interested only in the case $n\ge 5$, we omit the details.
\end{proof}
\begin{lemma}\label{Lem1bis}
If $\mathcal{M}$, $\mathcal{Z}$ and $\mathcal{N}$ are given as in \eqref{ImportantQuantitiesAbs}, and if
\begin{equation}\label{LocalBoundForSNorm}
\int_IN(t)^{4}dt\lesssim 1,
\end{equation}
then \eqref{Hyp3} implies the more general qualitative decay estimate:
\begin{equation}\label{Hyp3Stronger}
\lim_{A\to +\infty}\mathcal{M}(A)+ \mathcal{Z}(A)+\mathcal{N}(A)=0.
\end{equation}
\end{lemma}
\begin{proof}
Note that conservation of Mass, \eqref{LocalBoundForSNorm} and \eqref{BoundOnNonlinearTermByH} imply \eqref{Hyp2}. Similarly, Strichartz estimates with conservation of Mass for $u$ and \eqref{LocalBoundForSNorm} gives that
\begin{equation}\label{TrivialBound1ter}
\Vert u\Vert_{S^0(I)}\lesssim \Vert u\Vert_{L^\infty(I,L^2)}+\Vert F(u)\Vert_{N(I)}\lesssim 1.
\end{equation}
Using \eqref{TrivialBound1ter}, Sobolev's inequality and interpolation we get that
\begin{equation*}
\begin{split}
\mathcal{Z}(A)&\lesssim \Vert \vert\nabla\vert^\frac{n}{n+4} P_{\ge A}u\Vert_{L^\frac{2(n+4)}{n}([I,L^\frac{2(n+4)}{n+2})}\\
&\lesssim \Vert\vert\nabla\vert^\frac{n}{n+2}P_{\ge A}u\Vert_{L^\frac{2(n+2)}{n}(I,L^\frac{2(n+2)}{n})}^\frac{n+2}{n+4}\Vert P_{\ge A}u\Vert_{L^\infty(I,L^2)}^\frac{2}{n+4}\\
&\lesssim_u \mathcal{M}(A)^\frac{2}{n+4}.
\end{split}
\end{equation*}
While \eqref{Hyp2} and \eqref{NonlinearEstimates} give that
\begin{equation*}
\begin{split}
\mathcal{N}(A)&\lesssim_u \mathcal{Z}(A/8)^\frac{n+8}{n}+\mathcal{Z}(A/8)\left(\left(\sum_{N\le \sqrt{A}}+\sum_{\sqrt{A}\le N\le A}\right)\left(\frac{N}{A}\right)^\frac{n^2}{8(n+4)}\mathcal{Z}(N)\right)^\frac{8}{n}\\
&+\left(\left(\sum_{N\le \sqrt{A}}+\sum_{\sqrt{A}\le N\le A/8}\right)\left(\frac{N}{A}\right)^\frac{n\alpha}{n+8}\mathcal{Z}(N)\right)^\frac{n+8}{n}\\
&+\left(\left(\sum_{N\le \sqrt{A}}+\sum_{\sqrt{A}\le N\le A/8}\right)\left(\frac{N}{A}\right)\mathcal{Z}(N)\right)^\frac{7}{3}\\
&\lesssim_u \sup_{N\ge\sqrt{A}}\mathcal{Z}(N)\left(1+A^{-\frac{n}{2(n+4)}}\right)+A^{-\frac{\alpha}{2}}+A^{-\frac{7}{6}}.
\end{split}
\end{equation*} Estimate \eqref{Hyp3Stronger} is therefore a consequence of \eqref{Hyp3}.
\end{proof}
Our next result is a fundamental step allowing us to break the scaling under certain conditions.
\begin{lemma}\label{Lem2}
Suppose that $\mathcal{Z}$ and $\mathcal{N}$ are nonnegative functions satisfying \eqref{Hyp1}--\eqref{Hyp4}, \eqref{NonlinearEstimates} and \eqref{Hyp3Stronger}, then $\mathcal{Z}(A)\lesssim A^{-\frac{1}{4}}$.
\end{lemma}
\begin{proof}
First we claim that the following bound holds true. Let $0<\eta<1$, $\alpha=2(n+5)/(n+4)$. Then, if $A$ is sufficiently large depending on $u$ and $\eta$,
\begin{equation}\label{QTDE}\\
\begin{split}
\mathcal{Z}(A)&\le  \eta\mathcal{Z}(A/16)+KA^{-\frac{1}{2}}\\
&+K\left(\left(\sum_{N\le A/8}\left(\frac{N}{A}\right)^\frac{n\alpha}{n+8}\mathcal{Z}(N)\right)^\frac{n+8}{n}+\left(\sum_{N\le A/8}\left(\frac{N}{A}\right)\mathcal{Z}(N)\right)^\frac{7}{3}\right)
\end{split}
\end{equation}
for some constant $K>0$ independent of $\eta>0$. Let $0<\beta<1$. Using \eqref{Hyp4} then \eqref{Hyp2} and \eqref{NonlinearEstimates}, then \eqref{Hyp3Stronger}, we remark that
\begin{equation*}
\begin{split}
&\mathcal{Z}(A)\\
&\lesssim A^{-\frac{1}{2}}+\mathcal{N}(A/2)\\
&\lesssim A^{-\frac{1}{2}}+\mathcal{Z}(A/16)\left(\mathcal{Z}(A/16)^\frac{8}{n}+\left(\sum_{N\le A^\beta}\left(\frac{N}{A}\right)^\frac{n^2}{8(n+4)}\right)^\frac{8}{n}+\mathcal{Z}(A^\beta)^\frac{8}{n}\right)\\
&+\left(\sum_{N\le A/8}\left(\frac{N}{A}\right)^\frac{n\alpha}{n+8}\mathcal{Z}(N)\right)^\frac{n+8}{n}+\left(\sum_{N\le A/8}\left(\frac{N}{A}\right)\mathcal{Z}(N)\right)^\frac{7}{3}\\
&\lesssim A^{-\frac{1}{2}}+\eta\mathcal{Z}(A/16)+\left(\sum_{N\le A/8}\left(\frac{N}{A}\right)^\frac{n\alpha}{n+8}\mathcal{Z}(N)\right)^\frac{n+8}{n}+\left(\sum_{N\le A/8}\left(\frac{N}{A}\right)\mathcal{Z}(N)\right)^\frac{7}{3}\\
\end{split}
\end{equation*}
provided that $A$ is chosen sufficiently large. Since the constant in the inequality above does not depend on $\eta$, we get \eqref{QTDE}.

We know that there exists $K>1$ and $A_0$ such that \eqref{QTDE} holds for $\eta=1/(2^{5\alpha}K)$ and $A\ge A_0$. Iterating \eqref{QTDE} and using \eqref{Hyp2}, we get that, whenever $p$ satisfies $2^{-4p}A\ge A_0$, there holds that
\begin{equation*}
\begin{split}
&\mathcal{Z}(A)\\
&\le K\left(\eta\mathcal{Z}(A/16)+A^{-\frac{1}{2}}+\left(\sum_{N\le A/8}\left(\frac{N}{A}\right)^\frac{n\alpha}{n+8}\mathcal{Z}(N)\right)^\frac{n+8}{n}+\left(\sum_{N\le A/8}\frac{N}{A}\mathcal{Z}(N)\right)^\frac{7}{3}\right)\\
&\le \left(K\eta\right)^2\mathcal{Z}(2^{-8}A)+K(1+4K\eta)A^{-\frac{1}{2}}\\
&+K(1+2^{4\alpha}K\eta)\left(\left(\sum_{N\le A/8}\left(\frac{N}{A}\right)^\frac{n\alpha}{n+8}\mathcal{Z}(N)\right)^\frac{n+8}{n}+\left(\sum_{N\le A/8}\frac{N}{A}\mathcal{Z}(N)\right)^\frac{7}{3}\right)\\
&\le \left(K\eta\right)^p\mathcal{Z}(2^{-4p}A)+A^{-\frac{1}{2}}K\sum_{j=0}^p(4K\eta)^j\\
&+K\left(\sum_{j=0}^p\left(2^{4\alpha}K\eta\right)^j\right)\left(\left(\sum_{N\le A/8}\left(\frac{N}{A}\right)^\frac{n\alpha}{n+8}\mathcal{Z}(N)\right)^\frac{n+8}{n}+\left(\sum_{N\le A/8}\frac{N}{A}\mathcal{Z}(N)\right)^\frac{7}{3}\right)\\
&\lesssim \left(K\eta\right)^p+A^{-\frac{1}{2}}+\left(\sum_{N\le A/8}\left(\frac{N}{A}\right)^\frac{n\alpha}{n+8}\mathcal{Z}(N)\right)^\frac{n+8}{n}+\left(\sum_{N\le A/8}\frac{N}{A}\mathcal{Z}(N)\right)^\frac{7}{3},
\end{split}
\end{equation*}
where the constant in the last inequality is independent of $p$. In particular, when $2^{-4p}A\in [A_0,2A_0]$, the first term is bounded by $(2A_0)^\frac{5\alpha}{4}A^{-\frac{5\alpha}{4}}$. Hence, there exists a constant $k$ such that
\begin{equation}\label{IntermediaryEstimateQTDE}
\begin{split}
&\mathcal{Z}(A)\\
&\le k\left(A^{-\frac{1}{2}}+\left(\sum_{N\le A/8}\left(\frac{N}{A}\right)^\frac{n\alpha}{n+8}\mathcal{Z}(N)\right)^\frac{n+8}{n}+\left(\sum_{N\le A/8}\frac{N}{A}\mathcal{Z}(N)\right)^\frac{7}{3}\right)
\end{split}
\end{equation}
for all $A\ge A_0$. Now, fix $0<\delta<1$ small and let $A\ge A_0$ be a dyadic number sufficiently large so that $\mathcal{Z}(\sqrt{A})<\delta/2$, and $kA^{-\frac{1}{4}}<\delta$, and for $j\ge 0$, define $c_j=2^\frac{j}{4}\mathcal{Z}(2^jA)$. We claim that
\begin{equation}\label{finalClaimQTDE}
\forall j\ge 0, \hskip.1cm 0\le c_j\le \delta,\hskip.1cm\hbox{and}\hskip.1cm c_j\to 0\hskip.1cm\hbox{as}\hskip.1cmj\to +\infty.
\end{equation}
Note that this implies that $\mathcal{Z}(B)\lesssim B^{-\frac{1}{4}}$.
In order to prove \eqref{finalClaimQTDE}, we first remark that the claim holds for $c_0$, and, assuming the claims holds for $c_k$, $k\le j$, we get, using \eqref{Hyp2} that
\begin{equation}\label{IntemediaryEstimateBis}
\begin{split}
&\left(\sum_{N\le 2^{j-2}A}\left(\frac{N}{2^{j+1}A}\right)^\frac{n\alpha}{n+8}\mathcal{Z}(N)\right)^\frac{n+8}{n}\\
&\lesssim_u \left(\sum_{N\le \sqrt{A}}\left(\frac{N}{2^{j+1}A}\right)^\frac{n\alpha}{n+8}\right)^\frac{n+8}{n}+\left(\sum_{\sqrt{A}\le N\le A}\left(\frac{N}{2^{j+1}A}\right)^\frac{n\alpha}{n+8}\delta\right)^\frac{n+8}{n}\\
&+\left(\sum_{p=0}^{j-2}2^{(p-j-1)\frac{n\alpha}{n+8}}2^{-\frac{p}{4}}c_p\right)^\frac{n+8}{n}\\
&\lesssim A^{-\frac{\alpha}{2}}2^{-(j+1)\alpha}+2^{-(j+1)\alpha}\delta^\frac{n+8}{n}+2^{-\frac{(j+1)(n+8)}{4n}}\delta^\frac{n+8}{n}\\
&\le \delta^\frac{n+6}{n}2^{-\frac{j}{4}\frac{n+8}{n}}
\end{split}
\end{equation}
provided that $\delta>0$ is chosen sufficiently small so that $3\delta^\frac{2}{n}k_u<1$ where $k_u$ is the (universal) constant in the second inequality, and $A$ sufficiently large so that $A^{-\frac{\alpha}{2}}\le \delta^\frac{n+8}{n}$. Proceeding similarly, we also obtain that
\begin{equation}\label{IntemediaryEstimateTer}
\begin{split}
\left(\sum_{N\le 2^{j-2}A}\frac{N}{2^{j+1}A}\mathcal{Z}(N)\right)^\frac{7}{3}
\le \delta^2 2^{-\frac{j+1}{4}\frac{7}{3}}
\end{split}
\end{equation}
provided that $A$ is sufficiently large and $\delta$ is sufficiently small. Finally,
using \eqref{IntermediaryEstimateQTDE}, \eqref{IntemediaryEstimateBis} and \eqref{IntemediaryEstimateTer}, we conclude that
\begin{equation*}
\begin{split}
c_{j+1}&\le k2^\frac{j+1}{4}\left(2^{j+1}A\right)^{-\frac{1}{2}}+k\delta^\frac{n+6}{n}2^{-\frac{2j}{n}}+k\delta^22^{-j-1}\\
&\le \delta\left(A^{-\frac{1}{4}}+k\delta^\frac{6}{n}2+k\delta\right)2^{-\frac{j}{4n}}
\end{split}
\end{equation*}
hence, if $\delta>0$ is chosen sufficiently small, we see that $c_{j+1}\le\delta$, and that $c_{j+1}\to 0$ as $j\to +\infty$. This concludes the proof of Claim \eqref{finalClaimQTDE}, and hence the proof of Lemma \ref{Lem2}.
\end{proof}

Finally, we complete this section with a bootstrap argument that proves that whenever the scaling is broken in the sense that we have a small nonzero decay in $\mathcal{Z}$, the small decay is automatically upgraded to a stronger decay, and ultimately gives us the gain of two derivatives we are looking for.

\begin{lemma}\label{Lem3}
Suppose hypothesis \eqref{Hyp1}, \eqref{Hyp4}, \eqref{Hyp5} and \eqref{ImportantQuantitiesAbs} hold and that
$\mathcal{Z}$ satisfies $\mathcal{Z}(A)\lesssim A^{-\sigma_0}$ for some $\sigma_0>0$ and all dyadic $A$ sufficiently large. Then there holds that
\begin{equation}\label{GainOfRegularity}
\mathcal{M}(A)\lesssim A^{-\frac{7}{3}}+A^{-\alpha}.
\end{equation}
In particular, $\mathcal{M}(A)\lesssim A^{-2-\frac{1}{n+4}}$.
\end{lemma}
\begin{proof}
This is a consequence of the following claim: suppose that $\mathcal{Z}(A)\lesssim A^{-\sigma}$ for some $\sigma>0$, then
$$\mathcal{Z}(A)+\mathcal{N}(A)+\mathcal{M}(A)\lesssim A^{-\frac{n+8}{n}\sigma}+A^{-\sigma-\frac{n}{n+4}}+A^{-\alpha}+A^{-\frac{7}{3}}+A^{-\frac{7\sigma}{3}}.$$
We first prove the claim. Using \eqref{NonlinearEstimates}, we get that
\begin{equation*}
\begin{split}
\mathcal{N}(A)&\lesssim A^{-\frac{n+8}{n}\sigma}+A^{-\sigma-\frac{n}{n+4}}\left(\sum_{N\le A}N^\frac{n^2}{8(n+4)}\min(1,N^{-\sigma})\right)^\frac{8}{n}\\
&+A^{-\alpha}\left(\sum_{N\le A}N^\frac{n\alpha}{n+8}\min(1,N^{-\sigma})\right)^\frac{n+8}{n}+A^{-\frac{7}{3}}\left(\sum_{N\le A}\min(N,N^{1-\sigma})\right)^\frac{7}{3}\\
&\lesssim A^{-\frac{n+8}{n}\sigma}+A^{-\sigma-\frac{n}{n+4}}+A^{-\alpha}+A^{-\frac{7\sigma}{3}}+A^{-\frac{7}{3}}.
\end{split}
\end{equation*}
Then, using \eqref{Hyp5} and the inequality above, we get that
\begin{equation*}
\begin{split}
\mathcal{M}(A)\lesssim A^{-\frac{n+8}{n}\sigma}+A^{-\sigma-\frac{n}{n+4}}+A^{-\alpha}+A^{-\frac{7\sigma}{3}}+A^{-\frac{7}{3}}
\end{split}
\end{equation*}
and finally, combining the two estimates above with \eqref{Hyp1}, we get that
$$\mathcal{Z}(A)\lesssim A^{-\frac{n+8}{n}\sigma}+A^{-\sigma-\frac{n}{n+4}}+A^{-\alpha}+A^{-\frac{7\sigma}{3}}+A^{-\frac{7}{3}}$$
which concludes the proof of the claim. Iterating the claim a finite number of times, starting with $\sigma=\sigma_0>0$, we finish the proof of Lemma \ref{Lem3}.
\end{proof}

Combining Lemma \ref{Lem2} and \ref{Lem3}, we get the following result of gaining $2$ derivatives.
\begin{corollary}\label{GainOfRegCor}
Assuming \eqref{Hyp1}-\eqref{ImportantQuantitiesAbs}, we get that $\mathcal{M}(A)\lesssim A^{-2-\frac{1}{n+4}}$.
\end{corollary}


\section{Proof of Theorem \ref{MainThm}}\label{Sec-Proof}

In this section, we prove Theorem \ref{MainThm} after excluding all the possible scenarios in Theorem \ref{3ScenariosThm}, thanks to the analysis developed in Section \ref{Sec-Abs} above.

\subsection{The case of the Self-Similar solution}\label{Sec-Case1}

We start by excluding the case of a Self-Similar solution. In this favorable situation, we are able to work in full generality (i.e. in all dimensions, in the focusing and defocusing case).
\begin{proposition}[No Blow-up in finite time]\label{nonexistenceOfSelfSimilarBlowUp}
Let $n\ge 1$. Suppose that there exists a solution of \eqref{4NLS}, $u\in S^0(I)$ such that $I=(0,+\infty)$, $N(t)=t^{-\frac{1}{4}}$ and \eqref{CompactnessImage} holds true, then $M(u)\ge M(Q)$, and $\lambda<0$. In particular, there can be no Self-similar solution in the defocusing case.
\end{proposition}
Note in particular that we do not need $u$ to be radial. The proof will be the result of Lemmas \ref{QTD} and \ref{DecayInMass} below and Corollary \ref{GainOfRegCor}. Our approach is similar to the one in Killip, Tao and Visan \cite{KilTaoVis} and Killip, Visan and Zhang \cite{KilVisZha}.
We first need to introduce some more definitions. For $u$ satisfying the hypothesis of Proposition \ref{nonexistenceOfSelfSimilarBlowUp}, $F(z)=\lambda\vert z\vert^\frac{8}{n}z$, $v$ as in \eqref{CompactnessImage}  and $A$ a dyadic number, we let
\begin{equation}\label{ImportantQuantities}
\begin{split}
\mathcal{M}(A) &=\sup_{T>0}\Vert P_{\ge A}v(T)\Vert_{L^2}=\sup_{T>0}\Vert P_{\ge AT^{-\frac{1}{4}}}u(T)\Vert_{L^2}\\
\mathcal{Z}(A) &=\sup_{T>0}\Vert P_{\ge AT^{-\frac{1}{4}}}u\Vert_{Z([T,2T])}\\
\mathcal{N}(A) &=\sup_{T>0}\Vert P_{\ge AT^{-\frac{1}{4}}}F(u)\Vert_{N([T,2T])}.\\
\end{split}
\end{equation}
From compactness of $K$ in \eqref{CompactnessImage}, we get that
\begin{equation}\label{QDE2}
\lim_{A\to +\infty}\mathcal{M}(A)=\lim_{A\to+\infty}\sup_T\Vert P_{\ge AT^{-\frac{1}{4}}}u(T)\Vert_{L^2}\le \lim_{A\to+\infty}\sup_{v\in K}\Vert P_{\ge A}v\Vert_{L^2}=0.
\end{equation}
Hence \eqref{Hyp3} holds.
Strichartz estimates \eqref{StricEst} give \eqref{Hyp1}. By hypothesis,
\eqref{LocalBoundForSNorm} holds and we can use Lemma \ref{Lem1bis} to get \eqref{Hyp3Stronger},
\eqref{Hyp2}
and even, more generally that for all $T>0$,
\begin{equation}\label{TrivialBound1bis}
\Vert u\Vert_{S^0([T,2T])}\lesssim_u1.
\end{equation}
Furthermore, the estimate \eqref{NonlinearEstimates} is a consequence of Lemma \ref{Lem1} with the choice of $\mathcal{Z}$ and $\mathcal{N}$ in \eqref{ImportantQuantities}.

Now, we turn to the fundamental estimate of this subsection which gives us \eqref{Hyp4}.
\begin{lemma}\label{QTD}
Under the conditions above,
\begin{equation*}
\begin{split}
\mathcal{Z}(A)&\lesssim_u A^{-\frac{1}{2}}+\mathcal{N}(A/2)
\end{split}
\end{equation*}
for all dyadic $A$, i.e. \eqref{Hyp4} holds true.
\end{lemma}

\begin{proof}
Fix $T>0$ and let $I=[T,2T]$. Writing the Duhamel's formula \eqref{DuhamelFormula} with initial data at $T/2$ and using Strichartz estimates \eqref{StricEst}, we get that
\begin{equation}\label{QTDEProof1}
\begin{split}
&\Vert P_{\ge AT^{-\frac{1}{4}}}u\Vert_{Z(I)}\\
&\lesssim
\Vert P_{\ge AT^{-\frac{1}{4}}}e^{i(t-\frac{T}{2})\Delta^2}u(\frac{T}{2})\Vert_{Z(I)}+\Vert P_{\ge AT^{-\frac{1}{4}}}F(u)\Vert_{N([\frac{T}{2},2T])}.
\end{split}
\end{equation}
We estimate the second term by $\mathcal{N}(A/2)$. We now turn to the linear term.
Using conservation of Mass, Strichartz \eqref{StricEst} and Bernstein estimates \eqref{BernSobProp}, we get that
\begin{equation}\label{SLBound1}
\begin{split}
\Vert P_{\ge BT^{-\frac{1}{4}}}e^{i(t-\frac{T}{2})\Delta^2}u(\frac{T}{2})\Vert_{L^\frac{2(n+2)}{n}(I,L^\frac{2(n+2)}{n})}\lesssim_{M(u)} \left(BT^{-\frac{1}{4}}\right)^{-\frac{n}{n+2}}.
\end{split}
\end{equation}
Independently, using the Duhamel's formula \eqref{DuhamelFormula} with initial data at time $t=\varepsilon$, we get that
$$P_{BT^{-\frac{1}{4}}}e^{i(t-\frac{T}{2})\Delta^2}u(\frac{T}{2})=e^{i(t-\varepsilon)\Delta^2}P_{BT^{-\frac{1}{4}}}u(\varepsilon)+i\int_\varepsilon^\frac{T}{2}P_{BT^{-\frac{1}{4}}}e^{i(t-s)\Delta^2}F(u(s))ds.$$
We claim that for all $q\ge 2$, there holds that
\begin{equation}\label{LinearTermVanish}
\lim_{\varepsilon\to 0}\Vert e^{i(t-\varepsilon)\Delta^2}P_{BT^{-\frac{1}{4}}}u(\varepsilon)\Vert_{L^q([0,T]\times\mathbb{R}^n)}=0.
\end{equation}
Indeed, using the unitarity of the linear propagator and the properties of $u$ as in \eqref{CompactnessImage}, we get that
\begin{equation*}
\begin{split}
\Vert e^{i(t-\varepsilon)\Delta^2}P_{BT^{-\frac{1}{4}}}u(\varepsilon)\Vert_{L^\infty([0,T],L^2)}&=\Vert P_{BT^{-\frac{1}{4}}}u(\varepsilon)\Vert_{L^2}\\
&=\Vert P_{N(\varepsilon)^{-1}BT^{-\frac{1}{4}}}v(\varepsilon)\Vert_{L^2}\\
&\le \sup_{v\in K}\Vert P_{N(\varepsilon)^{-1}BT^{-\frac{1}{4}}}v\Vert_{L^2}\to 0,
\end{split}
\end{equation*}
as $\eps \to 0$. Then the H\"older's inequality in time gives \eqref{LinearTermVanish} for $q=2$, while Bernstein inequality \eqref{BernSobProp} gives \eqref{LinearTermVanish} for $q=\infty$. The general case follows by interpolation. This proves \eqref{LinearTermVanish}. Now to estimate the other term, we need to separate cases. We first treat the case $n\ge 5$. Using \eqref{DecayEstimate2} and \eqref{LinearTermVanish} we get
that for all $B,T$,
\begin{equation}\label{SLBound2}
\begin{split}
&\Vert P_{BT^{-\frac{1}{4}}}e^{i(t-\frac{T}{2})\Delta^2}u(\frac{T}{2})\Vert_{L^\frac{2(n+4)}{n-4}([T,2T],L^\frac{2(n+4)}{n-4})}\\
&=\lim_{\varepsilon\to 0}\Vert\int_\varepsilon^\frac{T}{2}e^{i(t-s)\Delta^2}P_{BT^{-\frac{1}{4}}}F(u(s))ds\Vert_{L^\frac{2(n+4)}{n-4}([T,2T],L^\frac{2(n+4)}{n-4})}\\
&\lesssim T^\frac{n-4}{2(n+4)}\lim_{\varepsilon\to 0}\Vert \int_\varepsilon^\frac{T}{2}e^{i(t-s)\Delta^2}P_{BT^{-\frac{1}{4}}}F(u(s))ds\Vert_{L^\infty([T,2T],L^\frac{2(n+4)}{n-4})}\\
&\lesssim T^\frac{n-4}{2(n+4)}\sup_{t\in [T,2T]}\int_0^\frac{T}{2}\frac{1}{\vert t-s\vert^{\frac{2n}{n+4}}}\Vert F(u(s))\Vert_{L^\frac{2(n+4)}{n+12}}\\
&\lesssim T^{-\frac{3n+4}{2(n+4)}}\Vert F(u)\Vert_{L^1([0,T],L^\frac{2(n+4)}{n+12})}\\
&\lesssim T^{-\frac{3n+4}{2(n+4)}}\sum_{\tau<T/2}\Vert u\Vert_{L^\frac{n+8}{n}([\tau,2\tau],L^\frac{2(n+4)(n+8)}{n(n+12)})}^\frac{n+8}{n}\\
&\lesssim T^{-\frac{3n+4}{2(n+4)}}\left(\sum_{\tau < T}\tau^\frac{n}{n+4}\Vert u\Vert_{L^\frac{(n+4)(n+8)}{4n}([\tau,2\tau],L^\frac{2(n+4)(n+8)}{n(n+12)})}^\frac{n+8}{n}\right)\\
&\lesssim_u T^{-\frac{3n+4}{2(n+4)}}\sum_{\tau < T}\tau^\frac{n}{n+4}=T^{-\frac{1}{2}}
\end{split}
\end{equation}
where the summation is over all $\tau=2^{-j}T$, $j\ge 0$, and we have used \eqref{TrivialBound1bis} to bound $u$ on $[\tau,2\tau]$. Combining \eqref{SLBound1} and \eqref{SLBound2} with H\"older's inequality, we get
\begin{equation*}
\begin{split}
&\Vert P_{ BT^{-\frac{1}{4}}}e^{i(t-\frac{T}{2})\Delta^2}u(\frac{T}{2})\Vert_{Z(I)}\\
&\lesssim \Vert P_{ BT^{-\frac{1}{4}}}e^{i(t-\frac{T}{2})\Delta^2}u(\frac{T}{2})\Vert_{L^\frac{2(n+4)}{n-4}(I\times\mathbb{R}^n)}^\frac{n}{3n+4}\Vert P_{BT^{-\frac{1}{4}}}e^{i(t-\frac{T}{2})\Delta^2}u(\frac{T}{2})\Vert_{L^\frac{2(n+2)}{n}(I\times\mathbb{R}^n)}^\frac{2(n+2)}{3n+4}\\
&\lesssim_u B^{-\frac{2n}{3n+4}}
\end{split}
\end{equation*}
and consequently, summing on all $B=2^jA$, $j\ge 0$,
\begin{equation*}
\begin{split}
\Vert P_{\ge AT^{-\frac{1}{4}}}e^{i(t-\frac{T}{2})\Delta^2}u(\frac{T}{2})\Vert_{Z(I)}&\le \sum_{B\ge A}\Vert P_{ BT^{-\frac{1}{4}}}e^{i(t-\frac{T}{2})\Delta^2}u(\frac{T}{2})\Vert_{Z(I)}\lesssim_u A^{-\frac{2n}{3n+4}}
\end{split}
\end{equation*}
which proves that the first term on the right hand side of \eqref{QTDEProof1} is also acceptable when $n\ge 5$. When $n\le 4$, we proceed as follows. Using \eqref{DecayEstimate} and \eqref{LinearTermVanish} we get that
\begin{equation}\label{SLBound3}
\begin{split}
&\Vert P_{BT^{-\frac{1}{4}}}e^{i(t-\frac{T}{2})\Delta^2}u(\frac{T}{2})\Vert_{L^\infty([T,2T]\times\mathbb{R}^n)}\\
&=\lim_{\varepsilon\to 0}\Vert\int_\varepsilon^\frac{T}{2}e^{i(t-s)\Delta^2}P_{BT^{-\frac{1}{4}}}F(u(s))ds\Vert_{L^\infty([T,2T]\times\mathbb{R}^n)}\\
&\lesssim \left(BT^{-\frac{1}{4}}\right)^{-n}\sup_{t\in I}\int_0^\frac{T}{2}\frac{1}{\vert t-s\vert^\frac{n}{2}}\Vert F(u(s))\Vert_{L^1}ds\\
&\lesssim B^{-n}T^{-\frac{n}{4}}\Vert u\Vert_{L^\frac{n+8}{n}([0,T]\times\mathbb{R}^n)}^\frac{n+8}{n}\\
&\lesssim B^{-n}T^{-\frac{n}{4}}\left(\sum_{\tau\le T/2}\tau^\frac{n}{8}\Vert u\Vert_{L^\frac{8(n+8)}{n(8-n)}([\tau,2\tau],L^\frac{n+8}{n})}^\frac{n+8}{n}\right)\\
&\lesssim_u B^{-n}T^{-\frac{n}{8}},
\end{split}
\end{equation}
where, once again, the summation is over $\tau=2^{-j}T$, and we have used \eqref{TrivialBound1bis} to bound $u$ on $[\tau,2\tau]$.
Combining \eqref{SLBound1} and \eqref{SLBound3} with H\"older's inequality, we get
\begin{equation*}
\begin{split}
&\Vert P_{ BT^{-\frac{1}{4}}}e^{i(t-\frac{t}{2})\Delta^2}u(\frac{T}{2})\Vert_{Z(I)}\\
&\lesssim \Vert P_{ BT^{-\frac{1}{4}}}e^{i(t-\frac{t}{2})\Delta^2}u(\frac{T}{2})\Vert_{L^\infty(I\times\mathbb{R}^n)}^\frac{2}{n+4}\Vert P_{ BT^{-\frac{1}{4}}}e^{i(t-\frac{t}{2})\Delta^2}u(\frac{T}{2})\Vert_{L^\frac{2(n+2)}{n}(I\times\mathbb{R}^n)}^\frac{n+2}{n+4}\\
&\lesssim_u B^{-\frac{3n}{n+4}}
\end{split}
\end{equation*}
and consequently,
\begin{equation*}
\begin{split}
\Vert P_{\ge AT^{-\frac{1}{4}}}e^{i(t-\frac{t}{2})\Delta^2}u(\frac{T}{2})\Vert_{Z(I)}&\le \sum_{B\ge A}\Vert P_{ BT^{-\frac{1}{4}}}e^{i(t-\frac{t}{2})\Delta^2}u(\frac{T}{2})\Vert_{Z(I)}\lesssim_u A^{-\frac{3n}{n+4}}
\end{split}
\end{equation*}
which proves that the first term in the right hand side of \eqref{QTDEProof1} is also acceptable when $n\le 4$.
\end{proof}

Now, we prove the last estimate we need.
\begin{lemma}\label{DecayInMass}
There holds that
\begin{equation*}\label{EstimateMass}
\mathcal{M}(A)\lesssim \sum_{k=0}^{\infty}\mathcal{N}(2^kA).
\end{equation*}
In particular, \eqref{Hyp5} holds true.
\end{lemma}

\begin{proof}
Fix $T>0$ and $t\in [T,2T]$.
Using the Duhamel's formula \eqref{DuhamelFormula} with initial data at time $S>0$, we get that
$$P_{\ge AT^{-\frac{1}{4}}}u(t)=P_{\ge AT^{-\frac{1}{4}}}e^{i(t-S)\Delta^2}u(S)+i\int_S^te^{i(t-s)\Delta^2}P_{\ge AT^{-\frac{1}{4}}}F(u(s))ds,$$
and using the fact that $N(t)=t^{-\frac{1}{4}}$, we get that
$$\Vert P_{\ge AT^{-\frac{1}{4}}}u(S)\Vert_{L^2}=\Vert P_{\ge A(S/T)^\frac{1}{4}}v(S)\Vert_{L^2}\to 0,\hskip.1cm\hbox{as}\hskip.1cm S\to +\infty, $$
where $v$ is as in \eqref{CompactnessImage}.
Consequently, we obtain that
\begin{equation*}
\begin{split}
\Vert P_{\ge AT^{-\frac{1}{4}}}u(t)\Vert_{L^2}&\lesssim \Vert P_{\ge AT^{-\frac{1}{4}}}u(2^{4L+4}T)\Vert_{L^2}\\
&+\sum_{k=0}^L\Vert\int_{2^{4k}T}^{2^{4(k+1)}T}e^{i(t-s)\Delta^2}P_{\ge AT^{-\frac{1}{4}}}F(u(s))ds\Vert_{L^2}\\
&\lesssim o(1)+\sum_{k=0}^L\Vert\int_{2^{4k}T}^{2^{4(k+1)}T}e^{-is\Delta^2}P_{\ge AT^{-\frac{1}{4}}}F(u(s))ds\Vert_{L^2},
\end{split}
\end{equation*}
where $o(1)\to 0$ as $L\to +\infty$. Letting $L\to +\infty$ in the estimate above then gives the result.
\end{proof}
Using the above results, we can now exclude the self-similar blow-up scenario.
\begin{proof}[Proof of Proposition \ref{nonexistenceOfSelfSimilarBlowUp}]
Let $u$ be a solution of \eqref{4NLS}, $u\in S^0_{loc}(I)$ on an interval $I=(0,+\infty)$, such that \eqref{CompactnessImage} holds true with $N(t)=t^{-\frac{1}{4}}$. In particular, $u$ blows up in finite time. With the preliminary remarks and Lemmas  \ref{QTD} and \ref{DecayInMass}, we deduce that \eqref{Hyp1}--\eqref{NonlinearEstimates} holds true, and consequently Corollary \ref{GainOfRegCor} gives that $u\in H^2$, and hence has a conserved energy. In the defocusing case, this gives a global bound on the $H^2$-norm, which by Proposition \ref{LocExProp} contradicts blow-up in finite time. In the focusing case, conservation of Energy and the sharp Gagliardo-Nirenberg inequality \eqref{Gag} give that, if $M(u)< M(Q)$
$$E(u_0)=E(u(t))\ge \frac{1}{2}\left(1-\frac{M(u)^\frac{8}{n}}{M(Q)^\frac{8}{n}}\right)\Vert\Delta u(t)\Vert_{L^2}^2> 0.$$ In particular, if $M(u)<M(Q)$, $u$ remains bounded in $H^2$ and does not blow up in finite time.
\end{proof}


\subsection{The case of Global solutions}\label{Sec-Case2}

Here we deal with the second and third scenarios of Theorem \ref{3ScenariosThm}. Our first result proves that the solution is actually smoother than expected.

\begin{proposition}\label{GainOfRegularityScenario23}
Let $n\ge 5$, and let $u$ be a solution of \eqref{4NLS} satisfying \eqref{CompactnessImage} with $N\le 1$.
Then $u(0)\in H^{2+\frac{1}{n+5}}$.
\end{proposition}

The proof of Proposition \ref{GainOfRegularityScenario23} is a consequence Section \ref{Sec-Abs} and of Lemmas \ref{DDFLemma} and \ref{EstimateForQ_0} which uses the ``Double Duhamel formula'' introduced in Tao \cite{Tao} which has proved to be helpful in many situations, see e.g., Killip and Visan \cite{KilVis, KilVis3}, Pausader \cite{PauBeam} and Tao \cite{Tao2}. We also refer to the survey by Killip and Visan \cite{KilVis2}.

\begin{lemma}\label{DDFLemma}
Let $u$ be a solution of \eqref{4NLS} such that $I=\mathbb{R}$, $h(t)\ge 1$ and \eqref{CompactnessImage} holds true, then for all $t\in \mathbb{R}$, there holds that
\begin{equation}\label{DF1}
\begin{split}
u(t)&=-i\int_t^{\to +\infty}e^{i(t-s)\Delta^2}F(u(s))ds\\
&=i\int_{\to-\infty}^te^{i(t-s)\Delta^2}F(u(s))ds
\end{split}
\end{equation}
where the integral are interpreted as a weakly convergent integral in $L^2$,
and
\begin{equation}\label{DDF}
\Vert P_{\ge A}u(t)\Vert_{L^2}^2=-\int_{s= -\infty}^t\int_{t^\prime=t}^{+\infty}\langle e^{i(t^\prime -s)\Delta^2}P_{\ge A}F(u(s)),P_{\ge A}F(u(t^\prime))\rangle dsdt^\prime
\end{equation}
where the integral is unconditionally convergent.
\end{lemma}

\begin{proof} Indeed, using the Duhamel's formula with initial data at time $T$ gives
$$u(t)=e^{i(t-T)\Delta^2}u(T)+i\int_T^te^{i(t-s)\Delta^2}F(u(s))ds,$$
and the first term weakly converges to $0$ in $L^2$.
Indeed from \eqref{DecayEstimate2}, we see that if $f_1,f_2$ are smooth compactly supported functions (in particular, $f_1,f_2\in L^1$), then
\begin{equation*}
\left\vert\langle e^{is\Delta^2}f_1,f_2\rangle\right\vert\lesssim_{f_1,f_2}\vert s\vert^{-\frac{n}{4}}.
\end{equation*}
Weak convergence follows by density.
This gives \eqref{DF1}. To get \eqref{DDF}, we also remark that, still by \eqref{DecayEstimate2},
for $f_1,f_2$ smooth functions compactly supported  in frequency, there holds that
\begin{equation}\label{ClaimWeakScattering}
\begin{split}
&\left\vert\langle e^{i(t^\prime-s)\Delta^2}P_{\ge A}g_{(N(s),y(s))}f_1,P_{\ge A}g_{(N(t^\prime),y(t^\prime))}f_2\rangle\right\vert\\
&\lesssim_{f_1,f_2}N(t^\prime)^{-\frac{n}{2}}N(s)^{-\frac{n}{2}}\vert t^\prime-s\vert^{-\frac{n}{4}}\\
&\lesssim_{f_1,f_2}\left(\frac{1}{N(t^\prime)\vert t^\prime\vert^\frac{1}{4}}\right)^\frac{n}{2}\left(\frac{1}{N(s)\vert s\vert^\frac{1}{4}}\right)^\frac{n}{2}\left(\frac{\sqrt{\vert t^\prime s\vert}}{\vert t^\prime-s\vert}\right)^\frac{n}{4} \hskip.1cm\hbox{and}\\
&=0\hskip.1cm \hbox{if}\hskip.1cm N(s)\lesssim_{f_1,A}1\hskip.1cm\hbox{or}\hskip.1cm \hbox{if}\hskip.1cm N(t^\prime)\lesssim_{f_2,A}1,\\
\end{split}
\end{equation}
where the last line follows from the fact that the Fourier support of $g_{(N(s),y(s))}f_1$ and $g_{(N(t^\prime),y(t^\prime)}f_2$ is included in $B(0,A/2)$ if $N(s)$ or $N(t^\prime)$ is too small.
From \eqref{ClaimWeakScattering}, we get that
\begin{equation}\label{ClaimWeakScattering2}
\left\vert\langle e^{i(t^\prime-s)\Delta^2}P_{\ge A}g_{(N(s),y(s))}f_1,P_{\ge A}g_{(N(t^\prime),y(t^\prime))}f_2\rangle\right\vert\to 0
\end{equation}
as $t^\prime-s\to \infty$ with $s\le t\le t^\prime$.
By compactness of $K$, we can replace $g_{(N(s),y(s))}f_1$ by $u(s)$ and $g_{(N(t),y(t))}f_2$ by $u(t^\prime)$ in \eqref{ClaimWeakScattering2}. Using this, we get that
\begin{equation*}
\begin{split}
&\int_{s=S}^t\int_{t^\prime=t}^T\langle e^{i(t^\prime -s)\Delta^2}P_{\ge A}F(u(s)),P_{\ge A}F(u(t^\prime))\rangle dsdt^\prime\\
&=\langle \int_{s=S}^te^{i(t-s)\Delta^2}P_{\ge A}F(u(s))ds,\int_{t^\prime=t}^Te^{i(t-t^\prime)\Delta^2}P_{\ge A}F(u(t^\prime))dt^\prime\rangle\\
&=-\langle \left(P_{\ge A}u(t)-e^{i(t-S)\Delta}P_{\ge A}u(S)\right),\left(P_{\ge A}u(t)-e^{i(t-T)\Delta}P_{\ge A}u(T)\right)\rangle\\
&=-\Vert P_{\ge A}u(t)\Vert_{L^2}^2+\langle e^{i(t-S)\Delta}P_{\ge A}u(S),P_{\ge A}u(t)\rangle+\langle P_{\ge A}u(t),e^{i(t-T)\Delta}P_{\ge A}u(T)\rangle\\
&+\langle e^{i(T-S)\Delta}P_{\ge A}u(S),P_{\ge A}u(T)\rangle,
\end{split}
\end{equation*}
and letting $(T,S)\to (+\infty,-\infty)$, and using \eqref{ClaimWeakScattering2} we obtain the result.
\end{proof}

Let $t$ be a chosen time, and choose $R>0$ a time scale (ultimately, we will choose $R=1$), then we divide
$$I_t=\{(s,t^\prime):-\infty< s\le t\le t^\prime< +\infty\}=Q_{-1}\cup Q_{\ge 0}$$
where
$$Q_{-1}=\{(s,t^\prime): t-R\le s\le t\le t^\prime\le t+R\}$$
correspond to the times close to $t$, and we make a Whitney decomposition of the times far away from $t$ as follows:
$$Q_{\ge 0}=\cup_{k\ge 0,i\in\{1,2,3\}}Q_k^i$$
with $Q_k^i$ the cubes of length $l_k=2^kR$, situated at distance greater or equal to $\frac{1}{\sqrt{2}}l_k$ of the diagonal $\{(s,s),s\in\mathbb{R}\}$ and such that
$$Q_{-1}\cup\cup_{k=0,i\in\{1,2,3\}}^KQ_k^i=\{(s,t^\prime):t-2^{K+1}R\le s\le t\le t^\prime\le t+2^{K+1}R\}.$$
More precisely, $Q_k^1$ is centered at $(t-2^{k-1}3R,t+2^{k-1}R)$, $Q_k^2$ at $(t-2^{k-1}3R,t+2^{k-1}3R)$ and $Q_k^3$ at $(t-2^{k-1}R,t+2^{k-1}3R)$.

Our next lemma shows that the contribution of large times (i.e. $Q_{\ge 0}$, with $R=1$) for large frequencies is small in dimensions $n\ge 5$.

\begin{lemma}\label{EstimateForQ_0}
Let $u$ satisfy the hypothesis of Proposition \ref{GainOfRegularityScenario23}, then there holds that
\begin{equation}\label{EForQ_0}
\left\vert\int\int_{Q_{\ge0}}\langle e^{i(t^\prime-s)\Delta^2}P_{\ge A}F(u(s)),P_{\ge A}F(u(t^\prime))\rangle dsdt^\prime\right\vert\lesssim_u \begin{cases}R^{\frac{4-n}{2}}A^{-n}&\hbox{if}\hskip.1cm 5\le n\le 8\\
R^{-2}A^{-8}&\hbox{if}\hskip.1cm n\ge 8
\end{cases}
\end{equation}
\end{lemma}

\begin{proof}
We first prove the results in dimensions $5\le n\le 8$. In this case, we remark that for any interval $I$, there holds that
\begin{equation}\label{L1BoundOnNonlinearity}
\Vert F(u)\Vert_{L^1(I,L^1)}\lesssim_u \langle\vert I\vert\rangle.
\end{equation}
Indeed, to get \eqref{L1BoundOnNonlinearity}, it suffices to prove that
$$\Vert F(u)\Vert_{L^1(I,L^1)}\lesssim_u 1$$
whenever $\vert I\vert\le 1$. but in this case, we have that
\begin{equation*}
\begin{split}
\Vert F(u)\Vert_{L^1(I,L^1)}&\lesssim \vert I\vert^\frac{n}{8}\Vert F(u)\Vert_{L^\frac{8}{8-n}(I,L^1)}\\
&\lesssim \vert I\vert^\frac{n}{8}\Vert u\Vert_{L^\frac{8(n+8)}{(8-n)n}(I,L^\frac{n+8}{n})}^\frac{n+8}{n}\\
&\lesssim \vert I\vert^\frac{n}{8}\Vert u\Vert_{S^0(I)}^\frac{n+8}{n}\\
&\lesssim_u 1
\end{split}
\end{equation*}
where in the last line, we have use \eqref{BoundOnNonlinearTermByH} and its consequence.
Then, using \eqref{DecayEstimate}, we get that, if $\vert t^\prime-s\vert\simeq l_k=R2^k$, letting $Q_k=I_k\times J_k$ and using \eqref{L1BoundOnNonlinearity}, we get that
\begin{equation*}
\begin{split}
&\left\vert\int\int_{Q_k}\langle e^{i(t^\prime-s)\Delta^2}P_{\ge A}F(u(s)),P_{\ge A}F(u(t^\prime))\rangle dsdt^\prime\right\vert\\
&\lesssim l_k^{-\frac{n}{2}}A^{-n}\int\int_{Q_k}\Vert F(u(s))\Vert_{L^1}\Vert F(u(t^\prime))\Vert_{L^1}\\
&\lesssim l_k^{-\frac{n}{2}}A^{-n}\Vert F(u)\Vert_{L^1(I_k,L^1)}\Vert F(u)\Vert_{L^1(J_k,L^1)}\\
&\lesssim_u l_k^{-\frac{n}{2}}A^{-n}\langle\vert I_k\vert\rangle\langle\vert J_k\vert\rangle\\
&\lesssim_u l_k^{\frac{4-n}{2}}A^{-n}
\end{split}
\end{equation*}
and since for each $k\ge 0$, there are exactly three such intervals, summing over $k$, we get \eqref{EForQ_0}. Now we turn to the case $n\ge 8$. In this case, the nonlinear term is in a Lebesgues space $L^{p^\prime}$ for $p^\prime=2n/(n+8)\ge 1$. Using \eqref{DecayEstimate} and conservation of Mass, we get that
\begin{equation*}
\begin{split}
&\left\vert\langle e^{i(t-s)\Delta^2}P_{\ge A}F(u(s)),P_{\ge A}F(u(t))\rangle\right\vert\\
&\lesssim \Vert e^{i(t-s)\Delta^2}P_{\ge A}F(u(t))\Vert_{L^p}\Vert P_{\ge A}F(u(s))\Vert_{L^{p^\prime}}\\
&\lesssim \vert t-s\vert^{-4}A^{-8}\Vert P_{\ge A}F(u(s))\Vert_{L^{p^\prime}}\Vert P_{\ge A}F(u(t))\Vert_{L^{p^\prime}}\\
&\lesssim_{M(u)} \vert t-s\vert^{-4}A^{-8}.
\end{split}
\end{equation*}
Integrating over $Q_k^i$, this gives
$$\left\vert\int\int_{Q_{k}^i}\langle e^{i(t^\prime-s)\Delta^2}P_{\ge A}F(u(s)),P_{\ge A}F(u(t^\prime))\rangle dsdt^\prime\right\vert\lesssim_u l_k^{-2}A^{-8}$$
and summing over all $k$, we get \eqref{EForQ_0} when $n\ge 8$. This ends the proof of Lemma \ref{EstimateForQ_0}.
\end{proof}
Lemmas \ref{DDFLemma} and \ref{EstimateForQ_0} give with Strichartz estimates that
\begin{equation}\label{EstimateForMass}
\Vert P_{\ge A}u(t)\Vert_{L^2}^2\lesssim_u A^{-5}+\Vert P_{\ge A}F(u)\Vert_{N([t-R,t])}\Vert P_{\ge A}F(u)\Vert_{N([t,t+R])}.
\end{equation}
To prove Proposition \ref{GainOfRegularityScenario23}, we introduce some more notations
\begin{equation}\label{DefMZN}
\begin{split}
&\mathcal{M}(A)=\Vert P_{\ge A}u\Vert_{L^\infty(\mathbb{R},L^2)}\\
&\mathcal{Z}(A)=\sup_I\Vert P_{\ge A}u\Vert_{Z(I)}\\
&\mathcal{N}(A)=\sup_{I}\Vert P_{\ge A}F(u)\Vert_{N(I)}
\end{split}
\end{equation}
where the supremum are taken on all intervals $I$ of length $\vert I\vert\le R$.
\begin{proof}[Proof of Proposition \ref{GainOfRegularityScenario23}]
We set $R=1$. With this choice of $\mathcal{M}$, $\mathcal{Z}$ and $\mathcal{N}$ in \eqref{DefMZN}, \eqref{Hyp1} follows from Strichartz estimates, and \eqref{Hyp2} follows from conservation of Mass, $N\le 1$, \eqref{BoundOnNonlinearTermByH} and Sobolev's inequality.
By the hypothesis on $N$, we get \eqref{LocalBoundForSNorm} from \eqref{BoundOnNonlinearTermByH} and also that  \eqref{Hyp3} is satisfied. Independently, \eqref{Hyp4} and \eqref{Hyp5} follow from \eqref{EstimateForMass} and \eqref{Hyp1}, while \eqref{NonlinearEstimates} is a consequence of Lemma \ref{Lem1} with the choice of $\mathcal{Z}$ and $\mathcal{N}$ in \eqref{DefMZN}. Applying Corollary \ref{GainOfRegCor}, we then get that $\mathcal{M}(A)\lesssim A^{-2-\frac{1}{n+4}}$, and summing over all frequencies, this gives that $u(0)\in H^{2+\frac{1}{n+5}}$. This ends the proof.
\end{proof}

This regularity result, combined with conservation of Energy and Mass allow to disprove the second scenario (i.e. a strong solution $u$ which remains compact up to rescaling cannot change scale). This is done in the following proposition.
\begin{proposition}\label{NoCascadeProp}
Let $n\ge 1$ and $u\in C(\mathbb{R},H^{2+\frac{1}{n+5}})$ be a solution of \eqref{4NLS} such that \eqref{CompactnessImage} holds true with $N\le 1$ and $u\ne 0$. Suppose there exists a sequence of times $t_k$ such that $N(t_k)\to 0$. Then $M(u)\ge M(Q)$ and $\lambda<0$. In particular, a high-to-low cascade scenario is not possible when $n\ge 5$, in the defocusing case or in the focusing case when $M_{max}<M(Q)$.
\end{proposition}

\begin{proof}
For $k\ge 0$, we split $u(t_k)$ into its high and low frequency components
$$u(t_k)=P_{\le L}u(t_k)+P_{>L}u(t_k)=u_l(t_k)+u_h(t_k).$$
By the hypothesis on $N$, we get that
\begin{equation*}
\lim_{k\to +\infty}\Vert P_{>L}u(t_k)\Vert_{L^2}\le \lim_{k\to+\infty}\sup_{v\in K}\Vert P_{>LN(t_k)^{-1}}v\Vert_{L^2}=0.
\end{equation*}
As a consequence, using Bernstein estimates \eqref{BernSobProp}, conservation of Mass and interpolation, we get that
\begin{equation*}
\begin{split}
\Vert \Delta u(t_k)\Vert_{L^2}
&\lesssim \Vert \Delta u_l(t_k)\Vert_{L^2}+\Vert\Delta u_h(t_k)\Vert_{L^2}\\
&\lesssim L^2M(u)+\Vert u_h(t_k)\Vert_{L^2}^\frac{1}{2n+11}\Vert u_h(t_k)\Vert_{H^{2+\frac{1}{n+5}}}^\frac{2(n+5)}{2n+11}\lesssim_{M(u)}L^2+o(1),
\end{split}
\end{equation*}
where $o(1)\to 0$ as $k\to+\infty$. Letting $k\to +\infty$ followed by $L\to 0$, we get that $\Vert \Delta u(t_k)\Vert_{L^2}\to 0$. Using conservation of Mass and the Gagliardo-Nirenberg inequality \eqref{Gag}, this gives that $E(u(t_k))\to 0$ as $k\to +\infty$, and by conservation of Energy, we conclude that $E(u)=0$. In the defocusing case, this is not possible if $0\ne u\in H^2$. In the focusing case, using the sharp Gagliardo-Nirenberg inequality \eqref{Gag} again, we conclude that $M(u)\ge M(Q)$. The last statement follows since Proposition \ref{GainOfRegularityScenario23} gives us the regularity required.
\end{proof}

To disprove the Soliton case, we use a Virial/Morawetz type of estimate.
For $e_1$ a unit vector in $\mathbb{R}^n$ we define the orthogonal Virial action along $e_1$ by
\begin{equation}\label{DefA}
A_R(t)=2\hbox{Im}\int_{\mathbb{R}^n}a(\frac{z_1}{R})z_1\partial_1u(t,x)\bar{u}(t,x)dx
\end{equation}
where $R>0$ will be chosen later on, $a\in C^\infty(\mathbb{R})$ is even and satisfies $a(x)=1$ for $\vert x\vert\le 1$, $a(x)=0$ for $\vert x\vert\ge 2$ and $a^\prime(x)\le 0$ for $x\ge 0$, $z=x-y(t)$ is the space variable in the frame moving with $u$, and $z_1$ is its coordinate in the direction of $e_1$. By \eqref{VariationOfH}, we have that $y(t)$ is a smooth function and $\vert \dot{y}\vert\lesssim_u1$. A trivial estimate gives that
\begin{equation}\label{BoundOnA}
\vert A_R(t)\vert\lesssim R \Vert u\Vert_{\dot{H}^1}^2\lesssim_u R
\end{equation}
uniformly in $t$, $R$.

\begin{proposition}\label{NoSolitonDefocusing}
Let $n\ge 1$ and $u\in C(\mathbb{R},H^{2+\frac{1}{n+5}})$ be a solution of \eqref{4NLS} such that \eqref{CompactnessImage} holds true with $N=1$.
Then
\begin{equation}\label{VirialArgument}
\partial_t A_R= 2\dot{y}_1e_1\cdot\hbox{Mom}(u)-16\int_{\mathbb{R}^n}\left(\frac{\vert\partial_1\nabla u\vert^2}{2}+\frac{\lambda}{2(n+4)}\vert u\vert^\frac{2(n+4)}{n}\right)dx+o_u(1)
\end{equation}
where $o_u(1)\to0$ as $R\to +\infty$ uniformly in $t$. In particular,
in the defocusing case $\lambda=1$, if $n\ge 5$, the Soliton-like scenario is not possible.
\end{proposition}
\begin{proof}
A similar proof appears in Pausader \cite{PauBeam2}. Since $u(0)\in H^2$, $u$ has a conserved momentum as in \eqref{DefOfMomentum}. Using that $u$ is bounded in $H^2$, integration by parts yields that
\begin{equation}\label{VirialEstProof}
\begin{split}
\partial_tA_R
&=-\dot{y}_12\hbox{Im}\int_{\mathbb{R}^n}\left(a^\prime\frac{z_1}{R}\partial_1u\bar{u}+a\partial_1u\bar{u}\right)dx\\
&-4\hbox{Im}\int_{\mathbb{R}^n}az_1\partial_1\bar{u}u_tdx-2\hbox{Im}\int_{\mathbb{R}^n}a\bar{u}u_tdx-2\hbox{Im}\int_{\mathbb{R}^n}a^\prime\frac{z_1}{R}\bar{u}u_tdx\\
&=2\dot{y}_1e_1\cdot\hbox{Mom}(u)-2\dot{y}\hbox{Im}\int_{\mathbb{R}^n}\left(a^\prime\frac{z_1}{R}\partial_1u\bar{u}+(a-1)\partial_1u\bar{u}\right)dx+O_u(\frac{1}{R})\\
&-16\int_{\mathbb{R}^n}\left(a+a^\prime\frac{z_1}{R}\right)\left(\frac{\vert\partial_1\nabla u\vert^2}{2}+\frac{\lambda}{2(n+4)}\vert u\vert^\frac{2(n+4)}{n}\right)dx\\
&=2\dot{y}_1e_1\cdot\hbox{Mom}(u)-16\int_{\mathbb{R}^n}\left(\frac{\vert\partial_1\nabla u\vert^2}{2}+\frac{\lambda}{2(n+4)}\vert u\vert^\frac{2(n+4)}{n}\right)dx+O_u(\frac{1}{R})\\
&-2\dot{y}\hbox{Im}\int_{\mathbb{R}^n}\left(a^\prime\frac{z_1}{R}\partial_1u\bar{u}+(a-1)\partial_1u\bar{u}\right)dx\\
&-16\int_{\mathbb{R}^n}\left(a+a^\prime\frac{z_1}{R}-1\right)\left(\frac{\vert\partial_1\nabla u\vert^2}{2}+\frac{\lambda}{2(n+4)}\vert u\vert^\frac{2(n+4)}{n}\right)dx.
\end{split}
\end{equation}
Independently, by \eqref{CompactnessImage}, we see that
\begin{equation*}\label{DefOfEpsilon}
\varepsilon(R)=\sup_{v\in K}\int_{\vert x\vert\ge R}\vert v(x)\vert^2dx=\sup_{t\in\mathbb{R}}\int_{\vert x-y(t)\vert\ge R}\vert u(t,x)\vert^2dx\to 0\hskip.1cm\hbox{as}\hskip.1cm R\to +\infty.
\end{equation*}
Consequently, using Sobolev's inequality and interpolation, we see that the last two lines in \eqref{VirialEstProof} above can be bounded by
\begin{equation*}
\begin{split}
(1+\vert\dot{y}\vert)\Vert (1-a+\vert a^\prime\frac{z}{R}\vert)u\Vert_{H^2}&\lesssim \Vert (1-a+\vert a^\prime\frac{z}{R}\vert)u\Vert_{L^2}^\frac{1}{2n+11}\Vert (1-a+\vert a^\prime\frac{z}{R}\vert)u\Vert_{H^\frac{2n+11}{n+5}}^\frac{2n+10}{2n+11}\\
&\lesssim_u \varepsilon(R)^\frac{1}{2n+11}\to 0
\end{split}
\end{equation*}
as $R\to +\infty$. This gives \eqref{VirialArgument}.
When $n\ge 2$, we can choose a vector $e_1$ which is orthogonal to $\hbox{Mom}(u)$. In this case the first term in the right hand side of \eqref{VirialArgument} vanishes. Now, suppose we are in the defocusing case $\lambda=1$, and $u$ is a Soliton-like solution as in Theorem \ref{3ScenariosThm}. Proposition \ref{GainOfRegularityScenario23} gives us the required regularity.
Then, choosing $R$ sufficiently large so that $\vert o_u(R)\vert\le \delta$ for some $\delta>0$ to be chosen below, and integrating, we get using \eqref{BoundOnNonlinearTermByH} that
\begin{equation*}
\begin{split}
\vert A_R(t) -A_R(0)\vert & \ge\int_0^t\left(\frac{8}{n+4}\int_{\mathbb{R}^n}\vert u(s,x)\vert^\frac{2(n+4)}{n}dx-\delta\right)ds\\
&\gtrsim \int_0^t h(s)^{-4}ds-C\delta t=(1-C\delta)t
\end{split}
\end{equation*}
for some constant $C$ depending only on $u$ and $n$. Choosing $\delta>0$ sufficiently small so that $C\delta<1$ and then $t$ sufficiently large, we contradict \eqref{BoundOnA}. Thus a soliton-like solution is not possible.
\end{proof}
Finally using the material developed above, we can finish the proof of Theorem \ref{MainThm}.
\begin{proof}[Proof of Theorem \ref{MainThm}] Using Proposition \ref{LocExProp}, we see that it suffices to prove that $M_{max}=+\infty$. Suppose it is not so. Then, applying Theorem \ref{3ScenariosThm}, we get that one of the 3 possible scenario in Theorem \ref{3ScenariosThm} holds. However, Propositions \ref{nonexistenceOfSelfSimilarBlowUp}, \ref{NoCascadeProp} and \ref{NoSolitonDefocusing} prove that this is not possible. Hence $$M_{max}=+\infty$$ and Theorem \ref{MainThm} is proved.
\end{proof}


\section{The Focusing case}\label{SecFoc}
In this section, we study the focusing variant of \eqref{4NLS}, that is, the case $\lambda=-1$.
\subsection{the Momentum and the translation parameter}
We first start with a proposition relating the translation function $y$, the Momentum and the Energy of a Soliton-like solution in the general case.
\begin{proposition}\label{SolitonImpliesNonZeroMomentum}
Let $n\ge 1$ and $u\in C(\mathbb{R},H^\frac{2n+11}{n+5})$ be a solution of \eqref{4NLS} such that \eqref{CompactnessImage} holds true with $N(t)=1$ and $y(0)=0$.
Then, for all $\varepsilon>0$, there exists a $C_{\varepsilon,u}$ such that
\begin{equation}\label{NonZeroMomentumEstimate}
\left\vert y(t)\cdot\hbox{Mom}(u)-8E(u)t \right\vert \le C_{\varepsilon, u}+\varepsilon t
\end{equation}
for all $t>0$. In particular, $y(t)\cdot\hbox{Mom}(u)\simeq t$ if $M(u)<M(Q)$.
\end{proposition}

\begin{proof} Choose an orthonormal basis $(e_1,\dots,e_n)$ of $\mathbb{R}^n$, and consider all the Virial actions corresponding to these vectors, and sum the corresponding contributions to \eqref{VirialArgument} to get
\begin{equation*}\label{VirialArgument2}
\partial_t A_R= 2\dot{y}\cdot\hbox{Mom}(u)-16\int_{\mathbb{R}^n}\left(\frac{\vert\Delta u\vert^2}{2}+\lambda\frac{n}{2(n+4)}\vert u\vert^\frac{2(n+4)}{n}\right)dx+o_u(1).
\end{equation*}
Integrating from $0$ to $t$ and using \eqref{BoundOnA}, we get \eqref{NonZeroMomentumEstimate}.
\end{proof}
From this we deduce that the equivalent of Theorem \ref{MainThm} also holds in the focusing case provided that the initial data is radially symmetrical and that the Mass of $u$ is below that of the Ground State, $M(u)<M(Q)$.

\begin{proof}[Proof of Theorem \ref{FocusingRadialThm}]
Indeed, in the radially symmetrical case, we get that failure of Theorem \ref{FocusingRadialThm} would imply that $M_{max}^{rad}<M(Q)$, and hence the existence of a solution $u\in S_{loc}(I)$ satisfying one of the three scenarios in Theorem \ref{3ScenariosThm} which is radially symmetrical. Proposition \ref{nonexistenceOfSelfSimilarBlowUp} shows that a self-similar blow-up is impossible. Proposition \ref{NoCascadeProp} shows that a cascade scenario is impossible, and Propositions \ref{GainOfRegularityScenario23} and \ref{SolitonImpliesNonZeroMomentum} shows that in the last scenario, there holds that $\hbox{Mom}(u)\ne 0$. But this contradicts the fact that $u$ is radially symmetrical. Hence none of the scenario in Theorem \ref{3ScenariosThm} is possible and Theorem \ref{FocusingRadialThm} holds true.
\end{proof}

\subsection{Motion of Mass}
In this subsection we study how the local Mass is dispersed by the solution, in the frame moving with the center of Mass to get the equality \eqref{InterVirEst} involving various quantities related to $u$. More precisely, we prove the following proposition

\begin{proposition}\label{AddedPropFoc}
Let $n\ge 1$ and $u\in C(\mathbb{R},H^\frac{2n+11}{n+5})$ be a solution of \eqref{4NLS} such that \eqref{CompactnessImage} holds true with $N(t)=1$ and $E(u)>0$.
Then, for all $\varepsilon>0$, and all intervals $I$ sufficiently large depending on $\varepsilon$, there holds that
\begin{equation}\label{InterVirEst}
2M(u)E(u)+\hbox{Mom}(u)\cdot\hbox{Im}\left(\frac{1}{\vert I\vert}\int_I\int_{\mathbb{R}^n}\nabla\bar{u}\Delta udx\right)=O(\varepsilon).
\end{equation}
In particular $u$ has nonzero Momentum and average Mass current.
\end{proposition}
The fact that $u$ has nonzero Momentum and Mass current tells us that $u$ is somewhat different from the Ground State and is used in the proof of Theorem \ref{NonradialFocusingTheorem} below.

\begin{remark}
The same conclusion follows formally from considering the {\it interaction Virial estimate}
$$V_i(t)=\iint_{\mathbb{R}^n\times\mathbb{R}^n}(x-y)\hbox{Im}(u(t,x)\nabla\bar{u}(t,x))\vert u(t,y)\vert^2dxdy,$$
instead of the local Mass \eqref{DefOfLocalMass} below. This interaction Virial estimate is related to the {\it interaction Morawetz estimate} used in the defocusing, Energy-critical case by Colliander, Keel, Staffilanni, Takaoka and Tao \cite{ColKeeStaTakTao}, Ryckman and Visan \cite{RycVis} and Visan \cite{Vis}.
\end{remark}

\begin{proof}
Without loss of generality, we can assume that $I=[0,T]$ and that the Momentum vector $\hbox{Mom}(u)$ is nonzero and parallel to the first vector $e_1$. We note $\hbox{Mom}(u)={\bf m}(u)e_1$ where ${\bf m}(u)>0$ is a positive quantity.
We define the local moment of Mass as follows: for $a$ as in \eqref{DefA} and $R>0$,
\begin{equation}\label{DefOfLocalMass}
M_R(t)=\int_{\mathbb{R}^n}a(\frac{z_1}{R})z_1\vert u(t,x)\vert^2dx,
\end{equation}
where $z=x-y(t)$.

This quantity allows us to understand how the Mass is (not) dispersed in a frame moving with the solution. We first remark that this local quantity is bounded. Indeed, by H\"older's inequality,
\begin{equation}\label{TrivialBoundOnMomMass}
\vert M_R(t)\vert\lesssim RM(u).
\end{equation}
uniformly in $R$ and $t$.
On the other hand, we can estimate the rate of change of $M_R$ by the following formula
\begin{equation}\label{RateOfMR}
\begin{split}
\partial_tM_R&=\dot{z}_1M(u)+\dot{z}_1\int_{\mathbb{R}^n}(a^\prime\frac{z_1}{R}-(1-a))\vert u\vert^2dx-2\hbox{Im}\int_{\mathbb{R}^n}a z_1\bar{u}\Delta^2udx\\
&=-\dot{y}_1M(u)-4\hbox{Im}\int_{\mathbb{R}^n}\partial_1\bar{u}\Delta udx\\
&+\frac{1}{R^2}\hbox{Im}\int_{\mathbb{R}^n}\left(2a^{\prime\prime\prime}\frac{z_1}{R}+6a^{\prime\prime}\right)\bar{u}\partial_1u dx\\
&+\int_{\mathbb{R}^n}\left((1-a)-a^\prime\frac{z_1}{R}\right)\left(-\dot{z}_1\vert u\vert^2+4\hbox{Im}\partial_1\bar{u}\Delta u\right)dx\\
&=-\dot{y}_1M(u)-4\hbox{Im}\int_{\mathbb{R}^n}\partial_1\bar{u}\Delta udx+\varepsilon(R),
\end{split}
\end{equation}
where $\varepsilon(R)\to 0$ as $R\to +\infty$, uniformly in $t$ by \eqref{CompactnessImage} and the fact that $u(t)$ is bounded in $H^{2+\frac{1}{n+5}}$. After normalizing $y(0)=0$, we integrate this between $0$ and $T$ to obtain
\begin{equation}\label{DiffInMR}
M_R(0)-M_R(T)=M(u)y_1(T)+4\hbox{Im}\int_0^T\int_{\mathbb{R}^n}\partial_1\bar{u}\Delta udx+O(T\varepsilon(R)).
\end{equation}
Let $\epsilon>0$. By Proposition \ref{SolitonImpliesNonZeroMomentum}, there exists $C_\epsilon>0$ such that
$$\vert y_1(t){\bf m}(u)-8E(u)t\vert\le C_\epsilon +\epsilon t.$$
Plugging this in \eqref{DiffInMR}, choosing $R>0$ sufficiently large so that $\varepsilon(R)\le\epsilon$, and dividing by $T$, we get
\begin{equation*}
8\frac{M(u)E(u)}{{\bf m}(u)}+4\hbox{Im}\frac{1}{T}\int_0^T\int_{\mathbb{R}^n}\partial_1\bar{u}\Delta udx=O(\epsilon)+\frac{1}{T}\left(O(R)+C_\epsilon\right).
\end{equation*}
Taking $T$ large enough, we can rewrite it as \eqref{InterVirEst}.
This concludes our proof.

\end{proof}


\subsection{An Inequality ``A la Banica''}
In this subsection, we conclude the proof of Theorem \ref{NonradialFocusingTheorem}. A big drawback of \eqref{4NLS} as compared to the usual Schr\"odinger equation is that the lack of Galilean invariance prevents us from normalizing our solution $u$ to have Momentum $0$, or to have mean position in the frequency space at the origin. As a consequence, we need to optimize the inequalities we rely on with respect to the position in frequency space. In this part, we precise the Sharp Gagliardo-Nirenberg inequality using an idea related to the work in Banica \cite{Ban}, and use it to reach a contradiction if $E(u)>0$.

For $u$ a complex function and $\phi$ a real smooth function, we compute
\begin{equation*}\label{Laplacian}
\begin{split}
\vert \Delta (u(x)e^{i\phi(x)})\vert^2
&=\vert\Delta u\vert^2-2\vert\nabla\phi\vert^2\hbox{Re}(\bar{u}\Delta u)+\vert u\vert^2(\vert\nabla\phi\vert^4+\vert\Delta\phi\vert^2)\\
&+4\vert\nabla u\cdot\nabla\phi\vert^2+2\Delta\phi\nabla\phi\nabla\vert u\vert^2\\
&+4\hbox{Im}(\Delta u\nabla\bar{u}\cdot\nabla\phi)+2\Delta\phi\cdot\hbox{Im}(\Delta u\bar{u})-4\vert\nabla\phi\vert^2\nabla\phi\cdot\hbox{Im}\left(u\nabla\bar{u}\right).
\end{split}
\end{equation*}
Integrating, we get
\begin{equation*}\label{PrecisedGag}
\begin{split}
\int_{\mathbb{R}^n}\vert\Delta(e^{i\phi}u)\vert^2=& \int_{\mathbb{R}^n}\vert\Delta u\vert^2+\int_{\mathbb{R}^n}\vert u\vert^2(\vert\nabla\phi\vert^4-\vert\Delta\phi\vert^2-2\nabla\Delta\phi\cdot\nabla\phi-2\Delta\vert\nabla\phi\vert^2)\\
&+2\int_{\mathbb{R}^n}\vert\nabla\phi\vert^2\vert\nabla u\vert^2+4\int_{\mathbb{R}^n}\vert\nabla u\cdot\nabla\phi\vert^2+4\hbox{Im}\int_{\mathbb{R}^n}(\Delta u\nabla\bar{u}\cdot\nabla\phi)\\
&-2\hbox{Im}\int_{\mathbb{R}^n}\nabla\Delta\phi\cdot\nabla u\bar{u}-4\int_{\mathbb{R}^n}\vert\nabla\phi\vert^2\nabla\phi\cdot\hbox{Im}\left(u\nabla\bar{u}\right)
\end{split}
\end{equation*}
for all $u$ and $\phi$.
and letting $\phi=Xx_1$, this gives
\begin{equation}\label{NonNegativityofPol}
\begin{split}
P_u(X)&=2E(e^{i\phi}u)\\
&=2E(u)-4\left(\hbox{Im}\int_{\mathbb{R}^n}(\Delta \bar{u}\partial_1u)dx\right)X\\
&+\left(2\int_{\mathbb{R}^n}\vert\nabla u\vert^2dx+4\int_{\mathbb{R}^n}\vert\partial_1 u\vert^2dx\right)X^2-4{\bf m}(u)X^3+M(u)X^4.
\end{split}
\end{equation}
Applying the sharp Gagliardo-Nirenberg to $e^{i\phi}u$ gives that
\begin{equation*}
\begin{split}
P_u=2E(e^{i\phi}u)&\ge (1-\kappa)\int_{\mathbb{R}^n}\vert\Delta (e^{i\phi}u)\vert^2dx\\
&\ge(1-\kappa)(P_u-2E(u)+\int_{\mathbb{R}^n}\vert\Delta u\vert^2dx),
\end{split}
\end{equation*}
 with $\kappa=(M(u)/M(Q))^\frac{8}{n}$.
Consequently, we get that
\begin{equation}\label{PosPu}
\kappa P_u\ge (1-\kappa)\left(\int_{\mathbb{R}^n}\vert\Delta u\vert^2dx-2E(u)\right).
\end{equation}

Now, we can finish the proof of Theorem \ref{NonradialFocusingTheorem}.

\begin{proof}[Proof of Theorem \ref{NonradialFocusingTheorem}]

Indeed, if $M_{max}<M(Q)$, and hence the existence of a solution $u\in S_{loc}(I)$ satisfying one of the three scenarios in Theorem \ref{3ScenariosThm}. Proposition \ref{nonexistenceOfSelfSimilarBlowUp} shows that a self-similar blow-up is impossible. Proposition \ref{NoCascadeProp} shows that a cascade scenario is impossible. Hence, we only need to exclude a Soliton. Proposition \ref{GainOfRegularityScenario23} and the sharp Gagliardo-Nirenberg inequality gives that in this case, $u\in C(\mathbb{R},H^\frac{2n+11}{n+5})$ and $E(u)>0$.

\medskip

Let $I_k$ be a sequence of intervals such that \eqref{InterVirEst} holds with $\varepsilon=1/k$, and consider the sequence of polynomials obtained by averaging \eqref{NonNegativityofPol} over $I_k$. More precisely, we consider
\begin{equation*}
\begin{split}
P_k=&2E(u)-4\left(\frac{1}{\vert I_k\vert}\hbox{Im}\int_{I_k}\int_{\mathbb{R}^n}(\Delta \bar{u}\partial_1u)dxdt\right)X\\
&+6\frac{1}{\vert I_k\vert}\left(\int_{I_k}\int_{\mathbb{R}^n}\left(\frac{1}{3}\vert\nabla u\vert^2+\frac{2}{3}\vert\partial_1 u\vert^2\right)dxdt\right)X^2-4{\bf m}(u)X^3+M(u)X^4.
\end{split}
\end{equation*}
All the coefficients of $P_k$ are uniformly bounded, and, by \eqref{PosPu}, the $P_k$ are nonnegative polynomials.
Without loss of generality, we can assume that the $P_k$ converge pointwise to a polynomial
$$P_\ast=2E(u)-4C(u)X+6G(u)X^2-4{\bf m}(u)X^3+M(u)X^4,$$
where
\begin{equation*}\label{DefCDG}
\begin{split}
&C(u)=\lim_{k\to+\infty}\frac{1}{\vert I_k\vert}\int_{I_k}\int_{\mathbb{R}^n}\hbox{Im}(\Delta \bar{u}(t,x)\partial_1u(t,x))dxdt,\\
&G(u)=\lim_{k\to +\infty}\frac{1}{\vert I_k\vert}\int_{I_k}\int_{\mathbb{R}^n}\left(\frac{1}{3}\vert\nabla u\vert^2+\frac{2}{3}\vert\partial_1u\vert^2\right)dxdt\\
\end{split}
\end{equation*}
and $C(u){\bf m}(u)=2E(u)M(u)$ by \eqref{InterVirEst}.
Now, we remark that
\begin{equation}\label{AddedEqt2}
\begin{split}
\frac{1}{2E(u)}P_\ast\left(\left(\frac{2E(u)}{M(u)}\right)^\frac{1}{4}\right)&= 6\left( \frac{G(u)}{\sqrt{2E(u)M(u)}}-1\right)-4\left(\sqrt{\alpha}-\frac{1}{\sqrt{\alpha}}\right)^2\\
&\le 6\left(\frac{G(u)}{\sqrt{2E(u)M(u)}}-1\right)
\end{split}
\end{equation}
for $\alpha={\bf m}(u)/(2E(u)M(u)^3)^\frac{1}{4}$.
Consequently, we get, after averaging \eqref{PosPu} over $I_k$ and letting $k\to +\infty$ and using \eqref{AddedEqt2} that
\begin{equation}\label{AddedEqt3}
6\left(\frac{G(u)}{\sqrt{2E(u)M(u)}}-1\right)\ge\frac{1-\kappa}{\kappa}\left(\frac{D(u)}{2E(u)}-1\right),
\end{equation}where
$$D(u)=\lim_{k\to+\infty}\frac{1}{\vert I_k\vert}\int_{I_k}\int_{\mathbb{R}^n}\vert\Delta u(t,x)\vert^2dxdt.$$
Letting $\Lambda\ge 1$ be such that $D(u)=\Lambda^2 2E(u)$ and remarking that, by H\"older's inequality $G(u)\le\sqrt{M(u)D(u)}$, we get with \eqref{AddedEqt3} that
$$6(\Lambda-1)\ge\frac{1-\kappa}{\kappa}(\Lambda^2-1),$$
which implies that
$$\frac{1-\kappa}{\kappa}\le \frac{6}{\Lambda+1}\le3,$$
and finally that $4\kappa\ge 1$.
This proves that
$$M(u)=M_{max}\ge \frac{1}{4^\frac{8}{n}}M(Q)$$ and Theorem \ref{NonradialFocusingTheorem} holds true.
\end{proof}

\noindent{\bf ACKNOWLEDGEMENT.}
B. Pausader expresses his deep thanks to Emmanuel Hebey for his constant support and for stimulating discussions during the early preparation of this work. S. Shao was supported by National Science Foundation under agreement No. DMS-0635607.

Any opinions, findings and conclusions or recommendations expressed in this paper are those of the authors and do not reflect necessarily the views of the National Science Foundation.

\end{document}